\documentclass[nosumlimits,twoside]{amsart}
\usepackage{amsfonts, amsmath, amssymb}
\usepackage[bookmarksnumbered,plainpages,driverfallback=dvipdfm,linktocpage=true, colorlinks=true, linkcolor=magenta, filecolor=magenta,urlcolor=cyan,citecolor=cyan]{hyperref}
\usepackage{graphicx, subcaption}
\usepackage{float}
\usepackage{srcltx}
\usepackage[all]{xy}
\usepackage{bbm}
\usepackage{version}
\usepackage[T1]{fontenc}
\usepackage{xcolor}
\usepackage{enumerate}
\usepackage[normalem]{ulem}
\usepackage{bm}
\usepackage{latexsym}
\usepackage{stmaryrd}
\usepackage[2emode]{psfrag}
\usepackage{yhmath}
\usepackage{array}
\usepackage{dsfont}
\usepackage{mathrsfs}
\usepackage{tikz-cd}
\usepackage{comment}
\usetikzlibrary { decorations.pathmorphing, decorations.pathreplacing, decorations.shapes }

\usepackage{quiver}

\newtheoremstyle{mytheorem}%
{10.0pt plus 2.0pt minus 2.0pt} 
{10.0pt plus 2.0pt minus 2.0pt} 
{\itshape} 
{} 
{\bfseries} 
{.} 
{ } 
{} 

\newtheoremstyle{mydefinition}%
{10.0pt plus 2.0pt minus 2.0pt} 
{10.0pt plus 2.0pt minus 2.0pt} 
{} 
{} 
{\bfseries} 
{.} 
{ } 
{} 

\newtheoremstyle{myremark}%
{10.0pt plus 2.0pt minus 2.0pt} 
{10.0pt plus 2.0pt minus 2.0pt} 
{} 
{} 
{\itshape} 
{.} 
{ } 
{} 

\theoremstyle{mytheorem}
\newtheorem{theorem}{Theorem}[section]
\newtheorem{lemma}[theorem]{Lemma}

\newtheorem{proposition}[theorem]{Proposition} 

\theoremstyle{myremark}
\newtheorem{remark}[theorem]{Remark}

\theoremstyle{mydefinition}
\newtheorem{definition}[theorem]{Definition}

\newtheoremstyle{myzusatz}
 {10.0pt plus 2.0pt minus 2.0pt} 
{10.0pt plus 2.0pt minus 2.0pt} 
{\itshape} 
{} 
{\bfseries} 
{.} 
{ } 
{\thmname{#1}\thmnumber{ #2}\thmnote{ #3}}

\theoremstyle{myzusatz}

\definecolor{gray1}{gray}{0.8}
\definecolor{gray2}{gray}{0.6}
\definecolor{gray3}{gray}{0.4}
\definecolor{gray4}{gray}{0.2}

\allowdisplaybreaks
\excludeversion{invisible?}
\excludeversion{proof?}

\DeclareMathOperator{\rbiprod}{{\cdot\kern-.33em\triangleright\kern-.43em<}}
\def\lbiprod{{>\!\!\!\triangleleft\kern-.33em\cdot}}

\newcommand{\Cc}{\mathcal{C}}

\newcommand{\Aa}{\mathcal{A}}

\newcommand{\ot}{\otimes}

{
\left\{\begin{aligned}}
{\end{aligned}
\right.
}

\begin{document}
\title{Hopf formulae for cocommutative Hopf algebras}
\author{Marino Gran and Andrea Sciandra}
\address{%
\parbox[b]{0.9\linewidth}{Institut de Recherche en Mathématique et Physique, Université Catholique de Louvain, Chemin du Cyclotron 2, B-1348 Louvain-la-Neuve, Belgium.}}
 \email{marino.gran@uclouvain.be}
\address{%
\parbox[b]{0.9\linewidth}{University of Turin, Department of Mathematics ``G.\@ Peano'', via
 Carlo Alberto 10, 10123 Torino, Italy.
 }}
 \email{andrea.sciandra@unito.it
}


\begin{abstract} 
The adjunction between coalgebras and Hopf algebras, first described by Takeuchi, allows one to prove that the semi-abelian category of cocommutative Hopf algebras has enough $\mathcal E$-projective objects with respect to the class $\mathcal{E}$ of cleft extensions. One then proves that, for any cocommutative Hopf algebra, there exists a weak $\mathcal{E}$-universal normal (=central) extension. This fact allows one to apply the methods of categorical Galois theory to classify normal $\mathcal{E}$-extensions and to provide an explicit description of the fundamental group of a cocommutative Hopf algebra in terms of a generalized Hopf formula. Moreover, with any cleft extension, we associate a 5-term exact sequence in homology that can be seen as a Hopf-theoretic analogue of the classical Stallings-Stammbach exact sequence in group theory.
\end{abstract}

\maketitle

\tableofcontents

\section{Introduction}
In this article we show that the methods of non-abelian homological algebra developed in recent years (see \cite{EGV, Protoadditive, RV, D-A}, for instance, and the references therein) can be effectively applied to the category $\mathsf{Hopf}_{\Bbbk,\mathrm{coc}}$ of cocommutative Hopf algebras over an arbitrary field $\Bbbk$.
As a matter of fact, even though semi-abelian categories \cite{JMT} include many interesting examples of categories that are not varieties of universal algebras, so far most of the applications of their (co)homology theories have been considered in the varietal context \cite{EGV}, with the only exceptions of the category of compact Hausdorff groups \cite{Protoadditive} and of some other categories of topological semi-abelian algebras \cite{BC}.
In fact, one can adapt the methods of non-abelian homological algebra to investigate the semi-abelian  category $\mathsf{Hopf}_{\Bbbk,\mathrm{coc}}$ of cocommutative Hopf algebras \cite{GSV}, mainly thanks to two crucial properties that $\mathsf{Hopf}_{\Bbbk,\mathrm{coc}}$ enjoys. 

The first one is the existence in $\mathsf{Hopf}_{\Bbbk,\mathrm{coc}}$ of a remarkable class $\mathcal E$ of extensions, that we call \emph{cleft extensions}, which have some good stability properties (Lemma \ref{thm:classE}), and turn out to play a similar role, from the point of view of categorical Galois Theory \cite{Jan-Algebra}, to the one played by the class of normal epimorphisms in the varietal context \cite{JK}. In our terminology, the cleft extensions are the normal epimorphisms in $\mathsf{Hopf}_{\Bbbk,\mathrm{coc}}$ with the additional property that they have a splitting when considered as morphisms of coalgebras. As we explain in Section \ref{preliminaries}, these morphisms are deeply related to the cleft extensions (see Definition \ref{cleft}) classically occurring in Hopf--Galois theory \cite{KT}, a theory that has already found several applications in different areas of mathematics, like noncommutative geometry and number theory. 

The second crucial property is the existence, in $\mathsf{Hopf}_{\Bbbk,\mathrm{coc}}$, of enough $\mathcal E$-projective objects (Lemma \ref{lem:projectiveobjects}), for this class $\mathcal E$ of cleft extensions. This fundamental fact can be seen as a direct consequence of the existence of the adjunction \eqref{adj} (recalled in Section \ref{Properties E}) between the categories of Hopf algebras and of coalgebras, that was first discovered by M. Takeuchi \cite{Ta}. 
These two properties are useful to build, for any cocommutative Hopf algebra $B$, a 
weak $\mathcal E$-universal normal extension $f \colon A \rightarrow B$ of $B$ (Proposition \ref{WeakUniversal}), so that one can define its fundamental group $\pi_1(B)$ as the group of automorphisms of $\mathbf{0}$ of the Galois groupoid of $f$, as in the work of G. Janelidze \cite{Jan} (see also the more recent article \cite{D-A}, and the references therein).

Thanks to this categorical approach and to the simple form of the commutators of normal (cocommutative) Hopf subalgebras (see \eqref{Huqcommutator}, for instance) 
we can establish, from any $\mathcal E$-projective presentation of a cocommutative Hopf algebra $B$, an explicit Hopf formula describing $\pi_1(B)$, that is, the second homology $\mathsf{H}_2(B)$ of $B$ (Theorem \ref{thm:fundamentalgroup} and Proposition \ref{Pi1}). Moreover, by adapting the methods developed by M. Duckerts-Antoine and T. Everaert in \cite{DE}, a new classification of normal $\mathcal{E}$-extensions of a cocommutative Hopf algebra $B$ can be proved in terms of some suitable discrete fibrations with codomain the Galois groupoid $\mathsf{Gal}(f)$ of any weak $\mathcal E$-universal normal extension $f$ of $B$ (Proposition \ref{discrete-fib}).

In the last section we observe that any short exact sequence in $\mathsf{Hopf}_{\Bbbk,\mathrm{coc}}$ 
\[\begin{tikzcd}
	\mathbf{0} & 
    \mathrm{Hker}(f) & A & B & \mathbf{0}
	\arrow[from=1-1, to=1-2]
	\arrow[hook, from=1-2, to=1-3]
	\arrow[from=1-3, to=1-4,"f"]
    \arrow[from=1-4, to=1-5],
\end{tikzcd}\]
where $f$ is a cleft extension, gives rise to a 5-term exact sequence 
\[\begin{tikzcd}
	\mathsf{H}_2(A) & \mathsf{H}_2(B) & \frac{\mathrm{Hker}(f)}{\mathrm{Hker}(f)[\mathrm{Hker}(f),A]^{+}} & \mathsf{H}_1(A) & \mathsf{H}_1(B) & \mathbf{0}
	\arrow[" ", from=1-2, to=1-3, ""]
	\arrow[" ",from=1-3, to=1-4," "]
    \arrow[" ", from=1-1, to=1-2,"\mathsf{H}_2(f) " ]
	\arrow[" ", from=1-4, to=1-5,"\mathsf{H}_1(f) " ]
    \arrow[" ",from=1-3, to=1-4," "]
    \arrow[" ",from=1-5, to=1-6," "],
\end{tikzcd}\]
involving the first and second homology Hopf algebras (Theorem \ref{SS}). This can be seen as a natural extension of the classical Stallings-Stammbach exact sequence for groups \cite{Stallings} (see \cite{EVdL, EG} for some other generalizations of this result).

We finally remark that the semi-abelianness of $\mathsf{Hopf}_{\Bbbk,\mathrm{coc}}$ and the fact that this category has enough $\mathcal E$-projectives make it certainly possible to further investigate cocommutative Hopf algebras from a (co)homological perspective.
Some other interesting structures, such as (cocommutative) Hopf braces 
\cite{AGV}, could also be studied by applying similar methods, since these are also semi-abelian \cite{Gran-Sciandra} and share many nice properties with cocommutative Hopf algebras.

\medskip

\noindent\textit{Notations and conventions}. We fix an arbitrary field $\Bbbk$ and all vector spaces are understood to be $\Bbbk$-vector spaces. Moreover, by a linear map we mean $\Bbbk$-linear, and the unadorned tensor product $\otimes$ stands for $\otimes_{\Bbbk}$. Algebras over $\Bbbk$ will be associative and unital, and coalgebras over $\Bbbk$ will be coassociative and counital. We will adopt Sweedler’s notation for calculations involving the comultiplication, so that we’ll write
$\Delta(x)=x_{1}\ot x_{2}$ omitting the summation. Given an algebra $A$, with $A^{\mathrm{op}}$ we mean the algebra with opposite product $x\cdot_{\mathrm{op}}y=yx$. Given a coalgebra $C$, with $C^{\mathrm{cop}}$ we mean the coalgebra with opposite coproduct $\Delta^{\mathrm{cop}}(x)=x_{2}\ot x_{1}$.

Given an object $X$ in a category $\Cc$, the identity morphism on $X$ will be denoted either by $\mathrm{Id}_{X}$ or $1_{X}$. The zero object of a pointed category will be denoted by $\mathbf{0}$. Given a morphism $f:A\to B$ in $\Cc$, the kernel of $f$ in $\Cc$ will be denoted by $\mathsf{ker}(f):\mathsf{Ker}(f)\to A$, 
while the kernel pair of $f$ will be denoted by $(\mathsf{Eq}(f),\pi_{1},\pi_{2})$.

\section{Preliminaries}\label{preliminaries}

In this section, we recall some fundamental notions and results that will be useful in the following. \medskip

\noindent\textbf{Categorical Galois structures}. We'll consider the special case of Galois structures associated with \textit{reflective subcategories}, (i.e.  subcategories such that the inclusion functor has a left adjoint) of semi-abelian categories \cite{JK}, where the class of morphisms considered consists in \emph{regular epimorphisms} (i.e. coequalizers of some pair of morphisms in the category). For more details, we refer the reader to \cite{Jan-Algebra, JK}. 

\begin{definition}
A \emph{Galois structure} is a system $\Gamma = \{\mathcal A, \mathcal B, I, H, \eta, \epsilon,\mathsf{RegEpi}(\Aa)\} $
where:
\begin{itemize}
    \item[1)] $\mathcal B$ is a full replete reflective subcategory of a semi-abelian category $\mathcal A$ with inclusion functor $H$ and reflector $I$, unit $\eta$, counit $\epsilon$ (which is then an isomorphism);
    \item[2)] $\mathsf{RegEpi}(\Aa)$ is the class of regular epimorphisms in $\mathcal A$.
\end{itemize}
For every $B \in \mathcal A$, such a Galois structure $\Gamma$ induces an adjunction 
\[\begin{tikzcd}
	{\mathcal A}\downarrow B && {\mathcal B}\downarrow I(B)
	\arrow[""{name=0, anchor=center, inner sep=0}, shift left=1, curve={height=-6pt}, from=1-1, to=1-3, "I^{B}"]
	\arrow[""{name=1, anchor=center, inner sep=0}, shift left=1, curve={height=-6pt}, from=1-3, to=1-1, "H^{B}"]
	\arrow["\dashv"{anchor=center, rotate=-90}, draw=none, from=0, to=1]
\end{tikzcd}\]
where: 
\begin{itemize}
    \item ${\mathcal A}\downarrow B $ is the full subcategory of the slice category $\mathcal A /B$ whose objects are the regular epimorphisms in $\Aa$ 
    with codomain $B$;
    \item ${\mathcal B}\downarrow I(B)$ is the full subcategory of $\mathcal B /I(B)$ whose objects are the regular epimorphisms 
    with codomain $I(B)$;
    \item  $I^B:{\mathcal A}\downarrow B  \rightarrow  {\mathcal B}\downarrow I(B)$ is the functor sending a regular epimorphism  $f \in {\mathcal A}\downarrow B$ to $I(f) \in {\mathcal B}\downarrow I(B)$;
\item $H^B:{\mathcal B}\downarrow I(B)  \rightarrow {\mathcal A}\downarrow B $ is the functor sending a regular epimorphism  $\phi:X\to I(B) \in {\mathcal B}\downarrow I(B) $ to the first projection $\pi_1 \in {\mathcal A}\downarrow B$ in the pullback
\[\begin{tikzcd}
	B \times_{HI(B)} H(X) & H(X) \\
	B & HI(B)
	\arrow[" ",{name=0, anchor=center, inner sep=0}, from=1-1, to=1-2, "{\pi_2}"]
	\arrow[" ",{name=0, anchor=center,  sep=0}, from=1-1, to=2-1, "\pi_1"']
	\arrow[" ",{name=0, anchor=center,  sep=0}, from=1-2, to=2-2, "H(\phi)"]
	\arrow[" ",{name=0, anchor=center, inner sep=0}, from=2-1, to=2-2," \eta_B"']
\end{tikzcd}\]
\item for regular epimorphisms $f \colon A \rightarrow B$ and $\phi \colon X \rightarrow  I(B)$
\[
(\eta^B)_{f} = \langle f, \eta_A \rangle : A \rightarrow B \times_{HI(B)} H(A),\quad (\epsilon^B)_{\phi } = \epsilon_X I(\pi_2) : I( B \times_{HI(B)} H(X)) 
\longrightarrow X.
\]
\end{itemize}
\end{definition}
A Galois structure is \emph{admissible} \cite{JK} when, for any object $B$ in $\mathcal A$, the counit $\epsilon^B$ is an isomorphism, a property that is equivalent to the fact that the functor $H^B \colon {\mathcal B}\downarrow I(B)  \rightarrow {\mathcal A}\downarrow B$ is fully faithful.  \medskip

We also recall the notions of trivial and normal extensions associated with a Galois structures $\Gamma = \{\mathcal A, \mathcal B, I, H, \eta, \epsilon, \mathsf{RegEpi}(\Aa)\} $:

\begin{itemize}
    \item[1)] a regular epimorphism $f \colon A \rightarrow B$ 
    is a \emph{trivial extension} if the following commutative square 

\[\begin{tikzcd}
	A & B \\
	HI(A) & HI(B)
	\arrow[from=1-1, to=1-2, "f"]
	\arrow[from=1-1, to=2-1, "\eta_A"']
	\arrow[from=1-2, to=2-2, "\eta_B"]
	\arrow[from=2-1, to=2-2, "HI(f)"']
\end{tikzcd}\]
is a pullback; 
\item[2)] a regular epimorphism $f\colon A \rightarrow B$ 
is a
\emph{normal extension} if the projection $\pi_1$ in the following pullback is a trivial extension:

\[\begin{tikzcd} 
	\mathsf{Eq}(f) & A \\
	A & B
	\arrow[from=1-1, to=1-2, "\pi_2"]
	\arrow[from=1-1, to=2-1, "\pi_1"']
    \arrow["\lrcorner"{anchor=center, pos=0.125}, draw=none, from=1-1, to=2-2]
	\arrow[from=1-2, to=2-2, "f"]
	\arrow[from=2-1, to=2-2, "f"']
\end{tikzcd}\]
Note that the projection $\pi_1$ is a trivial extension if and only if $\pi_2$ is a trivial extension.
\end{itemize}
If $\Gamma$ is admissible, then $I$ preserves pullbacks along trivial extensions. In particular, trivial extensions are pullback-stable, so that every trivial extension is a normal extension \cite[Proposition 2.4]{JK2}.
It is well known that a normal extension $f$ is a \textit{central extensions}, i.e. there is a regular epimorphism $p$ 
such that the pullback $p^{*}(f)$ of $f$ along $p$ is a trivial extension. Moreover, when $\Aa$ is a semi-abelian category and the reflective subcategory $\mathcal B$ is a \emph{Birkhoff subcategory} (i.e. it is also closed in $\mathcal A$ under subobjects and regular quotients), then the central extensions actually coincide with the normal extensions (\cite[Theorem 4.8]{JK}). As we'll see in Section \ref{sec:classificationnormalextensions}, the admissibility of the Galois structure guarantees the validity of a classification theorem of normal extensions. \medskip

\noindent\textbf{Cleft extensions and crossed products}. We now recall the special subclass of \textit{Hopf--Galois extensions} constituted by the so-called \textit{cleft extensions} and the corresponding equivalent notion of crossed product algebra. We refer the reader to \cite{Mo} for more details. The class of cleft extensions in $\mathsf{Hopf}_{\Bbbk,\mathrm{coc}}$ will play a fundamental role in the application of the methods of categorical Galois theory.

\begin{definition}\label{cleft}
Let $H$ be a Hopf algebra, $(A,\rho:A\to A\ot H)$ a right $H$-comodule algebra, i.e. $A$ is an algebra and a right $H$-comodule such that $m_{A}$ and $u_{A}$ are right $H$-colinear, 
and $B:=A^{\mathrm{co}H}=\{a\in A\ |\ \rho(a)=a\ot1_{H}\}
$ the subalgebra of $A$ of right $H$-coinvariant elements. 

Then, $B\subseteq A$ is called a \textit{cleft extension} if there is a convolution invertible morphism $j:H\to A$ of right $H$-comodules, the so-called \textit{cleaving map}. If $j$ is also a morphism of algebras, then $B\subseteq A$ is called a \textit{trivial extension}.
\end{definition}


\begin{remark}
We recall that $B\subseteq A$ is a \textit{Hopf--Galois extension} if the canonical map $\mathsf{can}:A\otimes_{B}A\to A\otimes H,\ a\otimes_{B}a'\mapsto (a\ot1)\rho(a')$ is invertible. Clearly, a trivial extension is cleft by definition. Moreover, cleft extensions form a special subclass of Hopf--Galois extensions. Indeed, using the cleaving map $j$ one can build the inverse of $\mathsf{can}$ as $\mathsf{can}^{-1}(a\ot h):=aj^{-1}(h_{1})\ot_{B}j(h_{2})$, where $j^{-1}$ denotes the convolution inverse of $j$. 
\end{remark}

Cleft extensions can be interpreted in terms of crossed product algebras.

\begin{definition}
Let $H$ be a Hopf algebra and $B$ an algebra. We say that $H$ \textit{measures} $B$ if there exists a linear map $\rightharpoonup:H\ot B\to B$ such that 
\[
h\rightharpoonup1_{B}=\varepsilon(h)1_{B},\qquad h\rightharpoonup(bb')=(h_{1}\rightharpoonup b)(h_{2}\rightharpoonup b'). 
\]
Suppose that $H$ measures $B$. Then:
\begin{itemize}
    \item[i)] a convolution invertible map $\sigma:H\ot H\to B$ is called a 2-\textit{cocycle with values in} $B$ if $\sigma(h\ot1_{H})=\varepsilon(h)1_{B}=\sigma(1_{H}\ot h)$ and
\[
(h_{1}\rightharpoonup\sigma(h'_{1}\ot h''_{1}))\sigma(h_{2}\ot h'_{2}h''_{2})=\sigma(h_{1}\ot h'_{1})\sigma(h_{2}h'_{2}\ot h'');
\]
\item[ii)] $B$ is a $\sigma$-\textit{twisted left $H$-module} if $1_{H}\rightharpoonup b=b$ and there exists a 2-cocycle $\sigma$ with values in $B$ such that 
\[
h\rightharpoonup(h'\rightharpoonup b)=\sigma(h_{1}\ot h'_{1})((h_{2}h'_{2})\rightharpoonup b)\sigma^{-1}(h_{3}\ot h'_{3});
\]
\item[iii)] if $B$ is a $\sigma$-twisted left $H$-module, the \textit{crossed product algebra} $B\#_{\sigma}H$ is the vector space $B\ot H$ considered as an algebra with product defined by
\[
(b\ot h)\cdot_{\#}(b'\ot h'):=b(h_{1}\rightharpoonup b')\sigma(h_{2}\ot h'_{1})\ot h_{3}h'_{2}
\]
and unit $1_{B}\ot1_{H}$.
\end{itemize}
\end{definition}

We recall that $B\#_{\sigma}H$ is a right $H$-comodule algebra with respect to the right $H$-coaction $\mathrm{Id}_{B}\ot\Delta_{H}:B\#_{\sigma}H\to(B\#_{\sigma}H)\ot H$ and we have $(B\#_{\sigma}H)^{\mathrm{co}H}=B\ot\Bbbk1_{H}\cong B$. Moreover, we observe that every left $H$-module algebra $B$ (where $\rightharpoonup$ is associative) is a $\sigma$-twisted left $H$-module with respect to the trivial 2-cocycle $h\ot h'\mapsto\varepsilon(hh')1_{B}$. In this case we obtain the \textit{smash product algebra} $B\#H$. \medskip

In the following theorem, we recall that crossed product algebras are in 1:1-correspondence with cleft extensions, where smash product algebras correspond to trivial extensions:

\begin{theorem}[{\cite[Theorem 7.2.2]{Mo}}]
    Any crossed product algebra $B\#_{\sigma}H$ is a cleft extension $B\subseteq B\#_{\sigma}H$ with cleaving map $j:H\to B\#_{\sigma}H$, $h\mapsto1_{B}\ot h$.

Conversely, given a cleft extension $B\subseteq A$ with cleaving map $j:H\to A$, one can define a $\sigma$-twisted left $H$-module action and a 2-cocycle with values in $B$ in the following way:
\[
h\rightharpoonup b:=j(h_{1})bj^{-1}(h_{2}), \qquad \sigma(h\ot h'):=j(h_{1})j(h'_{1})j^{-1}(h_{2}h'_{2}).
\]
Then, $A\cong B\#_{\sigma}H$ are isomorphic as right $H$-comodule algebras with inverse bijections:
\[
\psi(a):=a_{0}j^{-1}(a_{1})\ot a_{2}, \qquad \psi^{-1}(b\ot h):=bj(h).
\]
A cleft extension is a trivial extension if and only if the corresponding crossed product algebra is a smash
product algebra.
\end{theorem}

In the following, we are interested in the case when $B\#_{\sigma}H$ is a Hopf algebra. If also $B$ is a Hopf algebra and the morphisms $\rightharpoonup:H\otimes B\to B$ and $\sigma:H\otimes H\to B$ are coalgebra maps, 
then the crossed product algebra $B\#_{\sigma}H$ becomes a Hopf algebra with the tensor product coalgebra structure if and only if the following two compatibility conditions hold:
\begin{equation}\label{eq:BHHopf}
g_{1}\otimes(g_{2}\rightharpoonup a)=g_{2}\otimes(g_{1}\rightharpoonup a),\quad g_{1}h_{1}\otimes\sigma(g_{2}\otimes h_{2})=g_{2}h_{2}\otimes\sigma(g_{1}\otimes h_{1}),
\end{equation}
see \cite[Proposition 1.1]{ABM}. 

\begin{remark}
Notice that, if $H$ is cocommutative, then \eqref{eq:BHHopf} is trivially fulfilled. Therefore, if $B$ is a Hopf algebra and $\rightharpoonup$ and $\sigma$ are morphisms in $\mathsf{Coalg}_{\Bbbk}$ then $B\#_{\sigma}H$ is a Hopf algebra. The latter is cocommutative in case also $B$ is cocommutative. 
We also observe that, given a cleft extension $B\subseteq A$ where $A$ is a Hopf algebra, if $H$ is cocommutative and $j:H\to A$ is a morphism of coalgebras then $\rightharpoonup$ and $\sigma$ defined as in the previous theorem are also coalgebra maps.
\end{remark}

Assume that the previous conditions are satisfied so that $B\#_{\sigma}H$ is a Hopf algebra. Then $\pi_{H}:B\#_{\sigma}H\to H$, $b\ot h\mapsto\varepsilon(b)h$ is a morphism in $\mathsf{Hopf}_{\Bbbk}$ that has a section $i_{H}:H\to B\#_{\sigma}H$, $h\mapsto 1_{B}\ot h$ in $\mathsf{Coalg}_{\Bbbk}$. On the other hand, suppose that $\pi:A\to H$ is a morphism in $\mathsf{Hopf}_{\Bbbk}$ that has a section $i:H\to A$ in $\mathsf{Coalg}_{\Bbbk}$ such that $B:=A^{\mathrm{co}H}=\{a\in A\ |\ a_{1}\ot\pi(a_{2})=a\ot1_{H}\}$ is a Hopf algebra; $(A,\pi)$ is called a \textit{coalgebra split extension} of $B$ by $H$ in \cite[Definition 1.2]{ABM}. Then, $A$ is isomorphic to the Hopf algebra $B\#_{\sigma}H$, where 
\[
h\rightharpoonup b:=i(h_{1})bSi(h_{2}), \quad \sigma(g\ot h):=i(g_{1})i(h_{1})Si(g_{2}h_{2}), 
\]
see \cite[Proposition 1.3]{ABM}.

\begin{remark}\label{rmk:cleftextension}
    We recall that, given a morphism $\pi:A\to H$ in $\mathsf{Hopf}_{\Bbbk}$ with a section $i:H\to A$ in $\mathsf{Hopf}_{\Bbbk}$ (in particular, $i$ is in $\mathsf{Coalg}_{\Bbbk}$), then $B:=A^{\mathrm{co}H}$ is an object in $\mathsf{Hopf}(^{H}_{H}\mathcal{YD})$, hence it is not a standard Hopf algebra in general. This happens, for instance, under the assumption of cocommutativity: in this case $B=\mathsf{Ker}(\pi)$ in $\mathsf{Hopf}_{\Bbbk,\mathrm{coc}}$. Therefore, the class of morphisms in $\mathsf{Hopf}_{\Bbbk,\mathrm{coc}}$ that have a section in $\mathsf{Coalg}_{\Bbbk,\mathrm{coc}}$, that we denote by $\mathsf{SplitEpi}^{*}(\mathsf{Coalg}_{\Bbbk,\mathrm{coc}})$, precisely corresponds to the class of cleft extensions in $\mathsf{Hopf}_{\Bbbk,\mathrm{coc}}$ whose cleaving maps are morphisms of coalgebras.

 From now on we'll refer to the morphisms in $\mathsf{SplitEpi}^{*}(\mathsf{Coalg}_{\Bbbk,\mathrm{coc}}) $  as \emph{cleft extensions}.
    The class of split epimorphisms in $\mathsf{Hopf}_{\Bbbk,\mathrm{coc}}$ is denoted by $\mathsf{SplitEpi}(\mathsf{Hopf}_{\Bbbk,\mathrm{coc}})$.
\end{remark}

We can summarize the previous considerations on the category $\mathsf{Hopf}_{\Bbbk,\mathrm{coc}}$ with the following diagram:
\[
\begin{tikzcd}
	\{\textit{cleft extensions}\} & \{\textit{crossed product Hopf algebras}\} & \mathsf{SplitEpi}^{*}(\mathsf{Coalg}_{\Bbbk,\mathrm{coc}}) \\
	\{\textit{trivial extensions}\} & \{\textit{smash product Hopf algebras}\} & \mathsf{SplitEpi}(\mathsf{Hopf}_{\Bbbk,\mathrm{coc}})
	\arrow[tail reversed, from=1-1, to=1-2]
	\arrow[tail reversed, from=1-2, to=1-3]
	\arrow[hook, from=2-1, to=1-1]
	\arrow[tail reversed, from=2-1, to=2-2]
	\arrow[hook, from=2-2, to=1-2]
	\arrow[tail reversed, from=2-2, to=2-3]
	\arrow[hook, from=2-3, to=1-3]
\end{tikzcd}
\]
In the next section, the class of cleft extensions in $\mathsf{Hopf}_{\Bbbk,\mathrm{coc}}$ will be investigated from a categorical perspective.

\section{On the class of cleft extensions of cocommutative Hopf algebras}\label{Properties E}

We recall that the category $\mathsf{Hopf}_{\Bbbk,\mathrm{coc}}$ is semi-abelian and regular epimorphisms coincide with surjective maps, see \cite{GSV}. In this section, we study some categorical properties of a special subclass of regular epimorphisms in $\mathsf{Hopf}_{\Bbbk,\mathrm{coc}}$ using the adjunction between $\mathsf{Hopf}_{\Bbbk,\mathrm{coc}}$ and $\mathsf{Coalg}_{\Bbbk,\mathrm{coc}}$ obtained in \cite{Ta}, which we now briefly recall. \medskip

\noindent\textbf{The free Hopf algebra generated by a coalgebra}. Here we briefly recall the construction of the \textit{free Hopf algebra generated by a coalgebra} $C$ obtained in \cite{Ta}. Given a coalgebra $C$, consider the sequence of coalgebras $(V_{i})_{i\geq0}$ defined as follows
\[
V_{0}:=C, \qquad V_{i+1}:=V_{i}^{\mathrm{cop}}
\]
and the coalgebra $V=\bigoplus_{i=0}^{\infty}{V_{i}}$. Given the morphism of coalgebras $V\to V^{\mathrm{cop}}$, $(x_{0},x_{1},x_{2},\ldots)\mapsto(0,x_{0},x_{1},x_{2},\ldots)$ and the induced bialgebra map $S:T(V)\to T(V)^{\mathrm{op,cop}}$ one can define the two-sided ideal $I$ of $T(V)$ generated by the relations
\[
x_{1}S(x_{2})-\varepsilon(x)1,\qquad S(x_{1})x_{2}-\varepsilon(x)1,\quad \text{for all}\ x\in V.
\]
One can verify that $I$ is also a two-sided coideal and $S(I)\subseteq I$, and this implies that $F(C):=T(V)/I$ is a Hopf algebra \cite[Lemma 1]{Ta}. 
This provides the \textit{free functor} $F:\mathsf{Coalg}_{\Bbbk}\to\mathsf{Hopf}_{\Bbbk}$ which is left adjoint to the forgetful functor $U:\mathsf{Hopf}_{\Bbbk}\to\mathsf{Coalg}_{\Bbbk}$. We recall some details about this adjunction.

Let $\iota:C\rightarrow UF(C)$ be the morphism in $\mathsf{Coalg}_{\Bbbk}$ defined as the composite $C=V_{0}\hookrightarrow V\hookrightarrow T(V)\overset{\pi}\rightarrow F(C)$, where $\pi:T(V)\to F(C)$ denotes the canonical projection. 
The natural bijection 
\begin{equation}\label{naturaliso}
\Phi_{C,H}:=\mathrm{Hom}(\iota, H) :\mathsf{Hopf}_{\Bbbk}(F(C),H)\rightarrow\mathsf{Coalg}_{\Bbbk}(C,U(H)),\ f\mapsto f\iota
\end{equation}
characterizing the previous adjunction, has inverse 
\begin{equation}\label{canonicaliso}
\Phi_{C,H}^{-1}:\mathsf{Coalg}_{\Bbbk}(C,U(H))\rightarrow\mathsf{Hopf}_{\Bbbk}(F(C),H),
\end{equation}
described as in the proof of \cite[Lemma 1]{Ta}. More precisely, given $f:C\rightarrow U(H)$ in $\mathsf{Coalg}_{\Bbbk}$, define $f_{i}$ : $V_{i}\rightarrow H$ in $\mathsf{Coalg}_{\Bbbk}$, for $i\geq0$, as follows:
\[
f_{0}=f, \quad f_{i+1}=S_{H}f_{i}
\]
Then $\sum_{i\geq0}f_{i}:V\to H$ determines a morphism in $\mathsf{Coalg}_{\Bbbk}$ and induces a morphism $\varphi:T(V)\rightarrow H$ in $\mathsf{Bialg}_{\Bbbk}$. This map is zero on $I$ by the construction of $\sum_{i\geq0}f_{i}$. 
Therefore, a morphism $\Phi_{C,H}^{-1}(f):F(C)\rightarrow H$ in $\mathsf{Hopf}_{\Bbbk}$ is induced so that $\Phi_{C,H}^{-1}(f)\pi=\varphi$.

We have that $\iota=\eta_{C}$ is the $C$-component of the unit $\eta$ of this adjunction and it is an injective morphism of coalgebras, see \cite[Corollaries 6 and 9]{Ta}. We recall that $\eta_{C}$ monomorphism for all $C$ in $\mathsf{Coalg}_{\Bbbk}$ is equivalent to $F$ being faithful. 

The Hopf algebra $F(C)$ satisfies the following universal property: for any coalgebra map $\phi:C\to H$ with $H$ Hopf algebra there is a unique Hopf algebra map $\phi':F(C)\to H$ such that the following diagram commutes: 
\begin{equation}\label{universalproperty}
\begin{tikzcd}
	C && F(C) \\
	& H
	\arrow[from=1-1, to=1-3, "\eta_{C}"]
	\arrow[from=1-1, to=2-2, "\phi"']
	\arrow[dashed, from=1-3, to=2-2, "\phi'"]
\end{tikzcd}
\end{equation}

The free Hopf algebra $F(C)$ over a cocommutative coalgebra $C$ is cocommutative and $\mathsf{Hopf}_{\Bbbk,\mathrm{coc}}$ is monadic over $\mathsf{Coalg}_{\Bbbk,\mathrm{coc}}$, see \cite[Proposition 28 2.]{Porst}. Therefore, the previous adjunction restricts to the following one:
\begin{equation}\label{adj}
\begin{tikzcd}
	\mathsf{Coalg}_{\Bbbk,\mathrm{coc}} &&& \mathsf{Hopf}_{\Bbbk,\mathrm{coc}}
	\arrow[" "{name=1, anchor=center, inner sep=0}, curve={height=-8pt}, from=1-1, to=1-4,"F"]
	\arrow[" "{name=1, anchor=center, inner sep=0}, curve={height=-8pt}, from=1-4, to=1-1, "U"]
    \arrow["\vdash"{anchor=center, rotate=90}, draw=none, from=1-4, to=1-1],
\end{tikzcd}
\end{equation}
 We will use this adjunction in the following. 
\begin{remark}
The adjunction between coalgebras and Hopf algebras obtained in \cite{Ta} is related to the adjunction between bialgebras and Hopf algebras in \cite{Manin} (the free Hopf algebra generated by a bialgebra), as the former arises by combining the latter construction with the free bialgebra over a coalgebra, which is just given by the tensor algebra over the
underlying vector space of the coalgebra (see e.g. \cite{Sweedler}). We point out that, under the cocommutativity assumption, the construction of the free Hopf algebra over a bialgebra simplifies in some cases  \cite[Proposition 3.12]{AMS}.
\end{remark}

We denote by $\mathcal{E}$ the class of regular epimorphisms (equivalently, surjective maps) in $\mathsf{Hopf}_{\Bbbk,\mathrm{coc}}$ that have a section as morphisms in $\mathsf{Coalg}_{\Bbbk,\mathrm{coc}}$, that we refer to as \emph{cleft extensions} (see Remark \ref{rmk:cleftextension}).


\begin{remark}\label{rmk:crossedproduct}
As we recalled before, given $p:M\to N$ in $\mathcal{E}$ with section $i:N\to M$ in $\mathsf{Coalg}_{\Bbbk,\mathrm{coc}}$, the vector space $M^{\mathrm{co}N}:=\{x\in M\ |\ x_{1}\ot p(x_{2})=x\ot1_{N}\}$
coincides with the kernel $\mathsf{Ker}(p)$ of $p$ in $\mathsf{Hopf}_{\Bbbk,\mathrm{coc}}$, usually denoted by $\mathrm{Hker}(p)$. Moreover, by \cite[Proposition 1.3]{ABM}, this is equivalent to a crossed product extension $(\mathrm{Hker}(p)\#_{\sigma}N,\pi_{N})$, 
i.e. we have an isomorphism of Hopf algebras $\mathrm{Hker}(p)\#_{\sigma}N\cong M$.
\end{remark}

The forgetful functor $U:\mathsf{Hopf}_{\Bbbk,\mathrm{coc}}\to\mathsf{Coalg}_{\Bbbk,\mathrm{coc}}$ is conservative and this is equivalent to $\epsilon_{A}$ being a strong epimorphism for all $A$ in $\mathsf{Hopf}_{\Bbbk,\mathrm{coc}}$, where $\epsilon$ is the counit of the adjunction \eqref{adj}. Strong epimorphisms in $\mathsf{Hopf}_{\Bbbk,\mathrm{coc}}$ coincide with regular epimorphisms and then with surjective Hopf algebra maps. Therefore, we already know that the components of the counit $\epsilon$ of the adjunction \eqref{adj} are surjective; in the next result we show that they belong to $\mathcal{E}$. 

\begin{lemma}\label{lemma:counitadj}
    Given $A$ in $\mathsf{Hopf}_{\Bbbk,\mathrm{coc}}$, the $A$-component $\epsilon_{A}$ of the counit $\epsilon$ of the adjunction \eqref{adj} is 
    in $\mathcal{E}$.
\end{lemma}

\begin{proof}
   Using the explicit description of \eqref{canonicaliso}, we describe $\epsilon_{A}=\Phi^{-1}_{U(A),A}(\mathrm{Id}_{U(A)})$. We have the following sequence of morphisms in $\mathsf{Coalg}_{\Bbbk,\mathrm{coc}}$: $f_{0}=\mathrm{Id}_{U(A)}$, $f_{1}=S_{A}$, $f_{2}:= 
    S_{A}^{2}$ and so on. Since $A$ is cocommutative, we have $S_{A}^2=\mathrm{Id}_{A}$, hence the sequence $(f_{i})_{i\geq0}$ is given by $f_{i}=\mathrm{Id}_{U(A)}$ if $i$ is even and $f_{i}=S_{A}$ if $i$ is odd. Given the coalgebra $V:=\bigoplus_{i=0}^{\infty}{V_{i}}$ (where $V_{0}:=U(A)$ and $V_{i+1}:=V_{i}^{\mathrm{cop}}$ for all $i\geq0$), we obtain a 
    bialgebra map $\varphi:T(V)\to A$ induced by 
    $\sum_{i\geq 0}{f_{i}}:V\to A$ such that $\epsilon_{A}\pi=\varphi$, where $\pi:T(V)\to FU(A)$ is the canonical projection.
   Given $\eta_{U(A)}
   $ in $\mathsf{Coalg}_{\Bbbk,\mathrm{coc}}$, we know that $\eta_{U(A)}=\pi j$, where $j$ is the morphism $U(A)\hookrightarrow V\hookrightarrow T(V)$. Therefore, we have $\epsilon_{A}\eta_{U(A)}=\epsilon_{A}\pi j=\varphi j=\varphi|_{U(A)}=\mathrm{Id}_{A}$, hence $\epsilon_{A}$ has section $\eta_{U(A)}$ in $\mathsf{Coalg}_{\Bbbk,\mathrm{coc}}$. 
\end{proof}

We now give the following definitions, that will be useful througout the paper.

\begin{definition}\label{def:EprojectiveHopf}
    1). A cocommutative Hopf algebra $A$ is $\mathcal{E}$-\textit{projective} if for any morphism $\phi:A\to N$ in $\mathsf{Hopf}_{\Bbbk,\mathrm{coc}}$ and any morphism $p:M\to N$ in $\mathcal{E}$, there exists a morphism $\psi:A\to M$ in $\mathsf{Hopf}_{\Bbbk,\mathrm{coc}}$ such that $p\psi=\phi$:
\[\begin{tikzcd}
	& M \\
	A & N
	\arrow[from=1-2, to=2-2, "p"]
	\arrow[dashed, from=2-1, to=1-2, "\psi"]
	\arrow[from=2-1, to=2-2, "\phi"']
\end{tikzcd}\]
2). A short exact sequence 
\[
\begin{tikzcd}
	\mathbf{0} & \mathrm{Hker}(p) & P & B & \mathbf{0}
	\arrow[from=1-1, to=1-2]
	\arrow[hook, from=1-2, to=1-3, ""]
	\arrow[from=1-3, to=1-4,"p"]
	\arrow[from=1-4, to=1-5]
\end{tikzcd}
\]
in $\mathsf{Hopf}_{\Bbbk,\mathrm{coc}}$ is a $\mathcal E$-\textit{projective presentation} of $B$ if $P$ is $\mathcal E$-projective and $p$ is a morphism in $\mathcal E$. 
In particular, by Remark \ref{rmk:crossedproduct}, $P$ is isomorphic to the crossed product Hopf algebra $\mathrm{Hker}(p)\#_{\sigma}B$. We will sometimes denote a $\mathcal{E}$-projective presentation simply by $p:P\to B$. \medskip

\noindent 3). Given two $\mathcal E$-projective presentations $p:P\to B$ and $p':P'\to B'$, a \textit{morphism of $\mathcal{E}$-projective presentations} is a pair $(g,f)$ of morphisms in $\mathsf{Hopf}_{\Bbbk,\mathrm{coc}}$ making the following diagram commute
\begin{equation}\label{morphismpresentations}
\begin{tikzcd}
	\mathrm{Hker}(p) & P & B \\
	\mathrm{Hker}(p') & P' & B'
	\arrow[hook, from=1-1, to=1-2]
	\arrow[from=1-1, to=2-1, "\overline{g}"']
	\arrow[from=1-2, to=1-3, "p"]
	\arrow[from=1-2, to=2-2, "g"']
	\arrow[from=1-3, to=2-3, "f"]
	\arrow[hook, from=2-1, to=2-2]
	\arrow[from=2-2, to=2-3, "p'"']
\end{tikzcd}
\end{equation}
where $\overline{g}:=g|_{\mathrm{Hker}(p)}$. The composition of two morphisms of $\mathcal{E}$-projective presentations is clear.

We denote the category of $\mathcal{E}$-projective presentations in $\mathsf{Hopf}_{\Bbbk,\mathrm{coc}}$ and morphisms between them by $\mathsf{Pr}(\mathsf{Hopf}_{\Bbbk,\mathrm{coc}})$.
\end{definition}

\begin{remark}
    Given an arbitrary cocommutative Hopf algebra $A$, a morphism $\phi:A\to N$ in $\mathsf{Hopf}_{\Bbbk,\mathrm{coc}}$ and a morphism $p:M\to N$ in $\mathcal{E}$ with section $i:N\to M$ in $\mathsf{Coalg}_{\Bbbk,\mathrm{coc}}$, the morphism $i\phi:A\to M$ satisfies $pi\phi=\phi$. We observe that $i\phi$ is a morphism in $\mathsf{Coalg}_{\Bbbk,\mathrm{coc}}$ but it is in general not a morphism in $\mathsf{Hopf}_{\Bbbk,\mathrm{coc}}$.
\end{remark}

Using the free functor $F:\mathsf{Coalg}_{\Bbbk,\mathrm{coc}}\to\mathsf{Hopf}_{\Bbbk,\mathrm{coc}}$, one obtains $\mathcal{E}$-projective Hopf algebras, as it is shown by the following result.

\begin{lemma}\label{lem:projectiveobjects}
For any $C$ in $\mathsf{Coalg}_{\Bbbk,\mathrm{coc}}$, the free Hopf algebra $F(C)$ is $\mathcal{E}$-projective.
\end{lemma}

\begin{proof}
Let $C$ be an object in $\mathsf{Coalg}_{\Bbbk,\mathrm{coc}}$, $\phi:F(C)\to N$ a morphism in $\mathsf{Hopf}_{\Bbbk,\mathrm{coc}}$ and $p:M\to N$ a morphism in $\mathcal{E}$ with section $i:U(N)\to U(M)$ in $\mathsf{Coalg}_{\Bbbk,\mathrm{coc}}$. Then, using \eqref{naturaliso}, we obtain a morphism $\Phi_{C,N}(\phi):=\phi\eta_{C}:C\to U(N)$ in $\mathsf{Coalg}_{\Bbbk,\mathrm{coc}}$ 
and then $i\Phi_{C,N}(\phi)=i\phi\eta_{C}:C\to U(M)$ is a morphism in $\mathsf{Coalg}_{\Bbbk,\mathrm{coc}}$. 
By the universal property of $F(C)$ (see \eqref{universalproperty}), there exists a unique morphism $\psi:F(C)\to M$ in $\mathsf{Hopf}_{\Bbbk,\mathrm{coc}}$ such that $\psi\eta_{C}=i\phi\eta_{C}$. Therefore, we get $p\psi\eta_{C}=pi\phi\eta_{C}=\phi\eta_{C}$ and then, since $p\psi$ and $\phi$ are morphisms in $\mathsf{Hopf}_{\Bbbk,\mathrm{coc}}$, by uniqueness we obtain $p\psi=\phi$.
\end{proof}

The following result shows some stability properties of the class $\mathcal{E}$ of cleft extensions of Hopf algebras, that will be useful in the following. 

\begin{lemma}\label{thm:classE}
The following conditions are satisfied:
\begin{itemize}
    \item[1)] $\mathcal{E}$ contains the isomorphisms in $\mathsf{Hopf}_{\Bbbk,\mathrm{coc}}$;\medskip
    \item[2)] $\mathcal{E}$ is pullback stable; \medskip
    \item[3)] $\mathcal{E}$ is closed under composition; \medskip
    \item[4)] if $gf$ is in $\mathcal{E}$, then $g$ is also in $\mathcal{E}$. 
\end{itemize} 
\end{lemma}

\begin{proof}
    $1)$. Clearly all the isomorphisms in $\mathsf{Hopf}_{\Bbbk,\mathrm{coc}}$ are contained in $\mathcal{E}$. \medskip
    
    \noindent $2)$. The forgetful functor $U:\mathsf{Hopf}_{\Bbbk,\mathrm{coc}}\to\mathsf{Coalg}_{\Bbbk,\mathrm{coc}}$ in the adjunction \eqref{adj} preserves limits. Therefore, given a morphism $f$ in $\mathcal{E}$ and any $g$ in $\mathsf{Hopf}_{\Bbbk,\mathrm{coc}}$ the image by $U$ of the pullback of $f$ and $g$ in $\mathsf{Hopf}_{\Bbbk,\mathrm{coc}}$ is the pullback of $f$ and $g$ in $\mathsf{Coalg}_{\Bbbk,\mathrm{coc}}$. The property $2)$ then easily  follows,
    since $f$ is a split epimorphism in $\mathsf{Coalg}_{\Bbbk,\mathrm{coc}}$, and split epimorphisms are always stable under pullbacks.
\medskip

\noindent $3)$. Let $f:A\to B$ and $g:B\to C$ be in $\mathcal{E}$ so that there exist $i:B\to A$ and $i':C\to B$ in $\mathsf{Coalg}_{\Bbbk,\mathrm{coc}}$ such that $fi=\mathrm{Id}_{B}$ and $gi'=\mathrm{Id}_{C}$. Then, 
$gfii'=gi'=\mathrm{Id}_{C}$, where $ii':C\to A$ is a morphism in $\mathsf{Coalg}_{\Bbbk,\mathrm{coc}}$. Thus, $gf$ is in $\mathcal{E}$. \medskip

\noindent $4)$. Let $f:A\to B$ and $g:B\to C$ be in $\mathsf{Hopf}_{\Bbbk,\mathrm{coc}}$ such that $gf:A\to C$ is in $\mathcal{E}$, so there exists a morphism $i:C\to A$ in $\mathsf{Coalg}_{\Bbbk,\mathrm{coc}}$ such that $gfi=\mathrm{Id}_{C}$. Then, $fi:C\to B$ is a morphism in $\mathsf{Coalg}_{\Bbbk,\mathrm{coc}}$ and it is a section for $g$, so $g$ is in $\mathcal{E}$. 
\end{proof}


\section{Normal extensions of cocommutative Hopf algebras}

In this section, we apply the known results on normal extensions in a semi-abelian category to 
 the adjunction 

 \begin{equation}\label{Adj-abelianisation}
\begin{tikzcd}
\mathsf{Hopf}_{\Bbbk,\mathrm{coc}}	 &&& \mathsf{Hopf}_{\Bbbk,\mathrm{coc}}^{\mathrm{com}}
	\arrow[" "{name=1, anchor=center, inner sep=0}, curve={height=-8pt}, from=1-1, to=1-4,"\mathsf{ab}" ]
	\arrow[" "{name=1, anchor=center, inner sep=0}, curve={height=-8pt}, from=1-4, to=1-1, "U"]
    \arrow["\vdash"{anchor=center, rotate=90}, draw=none, from=1-4, to=1-1]
\end{tikzcd}
\end{equation}
where $\mathsf{Hopf}^{\mathrm{com}}_{\Bbbk,\mathrm{coc}}$ denotes the abelian category of commutative and cocommutative Hopf algebras, $U$ is the forgetful functor, and $\mathsf{ab}$ is the abelianisation functor 
that sends a cocommutative Hopf algebra $A$ to the quotient $\frac{A}{A[A,A]^{+}}$, where $[A,A]$ is the Hopf subalgebra of $A$ generated as an algebra by all the elements of the form  $a_1 b_1 S(a_2) S(b_2)$, for $a,b\in A$  \cite{GSV}:
\[
[A,A ] = \langle \{a_1 b_1 S(a_2) S(b_2) \mid a,b \in A\} \rangle_A.
\]
We recall that $[A,A]$ is exactly the largest Huq commutator  \cite{Huq} of $A$ and, for any normal Hopf subalgebras $X$ and $Y$ of $A$, the Huq commutator $[X,Y]$ is a normal Hopf subalgebra of $A$, as shown in \cite[Proposition 4.2]{GSV}.

The category $\mathsf{Hopf}_{\Bbbk,\mathrm{coc}}^{\mathrm{com}}$ is reflective
in $\mathsf{Hopf}_{\Bbbk,\mathrm{coc}}$ and it is also a Birkhoff subcategory of $\mathsf{Hopf}_{\Bbbk,\mathrm{coc}}$ \cite{JK}, since it is also
closed under Hopf subalgebras and regular quotients in $\mathsf{Hopf}_{\Bbbk,\mathrm{coc}}$. As shown in \cite{JK} any Birkhoff subcategory of an exact Mal'tsev category is admissible (i.e. the Galois structure associated is admissible), hence this is the case in particular for the  reflection \eqref{Adj-abelianisation}.

Thanks to the main result in \cite{Bourn-Gran} and to the fact that $\mathsf{Hopf}_{\Bbbk,\mathrm{coc}}$ is a semi-abelian category \cite{GSV}, one has the following characterization of normal extensions relative to the adjunction \eqref{Adj-abelianisation}: these are surjective morphisms $f \colon A \rightarrow B$ in $\mathsf{Hopf}_{\Bbbk,\mathrm{coc}}$ such that $[\mathsf{Ker} (f)=\mathrm{Hker}(f), A] = \mathbf{0}$, where 
\begin{equation}\label{Huqcommutator}
[\mathrm{Hker} (f), A] = \langle \{ k_1 a_1 S(k_2) S(a_2) \, \mid \, k \in \mathrm{Hker} (f), a \in A\}  \rangle_A 
\end{equation}
is the Huq commutator of the normal Hopf subalgebras 
$\mathrm{Hker}(f)=\{x\in A\ |\ x_{1}\ot f(x_{2})=x\ot1_{B}\}$
and $A$. 
By \cite[Lemma 4.1]{GSV}, we know that $k_{1}a_{1}S(k_{2})S(a_{2})=\varepsilon(a)\varepsilon(k)1$ for all $k\in\mathrm{Hker}(f)$ and $a\in A$ is equivalent to $ka=ak$ for all $k\in\mathrm{Hker}(f)$ and $a\in A$. Hence, we have the following result.

\begin{proposition}\label{existence-universal}
A surjective morphism $f \colon A \rightarrow B$ in $\mathsf{Hopf}_{\Bbbk,\mathrm{coc}}$ is a normal extension if and only if 
$\mathrm{Hker}(f)\subseteq Z(A)$, where $Z(A)$ denotes the center of $A$.
\end{proposition}




The unit $\eta$ of the adjunction \eqref{Adj-abelianisation} has components $\eta_{A}:A\to\frac{A}{A[A,A]^{+}}$, i.e. $\eta_{A}$ is the cokernel of the canonical inclusion $j:[A,A]\hookrightarrow A$ and $[A,A]=\mathsf{Ker}(\eta_{A})$. Clearly, the restriction of a morphism $f:A\to B$ in $\mathsf{Hopf}_{\Bbbk,\mathrm{coc}}$ to $[A,A]$ gives a morphism $f:[A,A]\to[B,B]$, hence the kernel of the unit $\eta$ 
is a subfunctor of the identity functor $\mathsf{Id}_{ \mathsf{Hopf}_{\Bbbk,\mathrm{coc}}} 
$ and we denote it as
\begin{equation}\label{def:functorkernelunit}
[\cdot] \colon \mathsf{Hopf}_{\Bbbk,\mathrm{coc}} \rightarrow \mathsf{Hopf}_{\Bbbk,\mathrm{coc}},\ A\mapsto [A]:=\mathrm{Hker}(\eta_{A}).
\end{equation}
In order to emphasize the functoriality of the operation of taking the kernels of the units we shall often use the notations $[A]$ and $[B]$ instead of $[A,A]$ and $[B,B]$.
Note also that this subfunctor $[\cdot]$ preserves regular epimorphisms, since the category $\mathsf{Hopf}^{\mathrm{com}}_{\Bbbk,\mathrm{coc}}$ is a Birkhoff subcategory of the semi-abelian category $\mathsf{Hopf}_{\Bbbk,\mathrm{coc}}$.
Given $f:A\to B$ in $\mathsf{Hopf}_{\Bbbk,\mathrm{coc}}$, we can consider the kernel pair of $f$ in $\mathsf{Hopf}_{\Bbbk,\mathrm{coc}}$
\begin{equation}\label{kernelpair}
\begin{tikzcd}
	\mathsf{Eq}(f) && {A}
	\arrow[shift left=2, from=1-1, to=1-3, "\pi_{1}"]
	\arrow[shift right=2, from=1-1, to=1-3, "\pi_{2}"']
	\arrow[from=1-3, to=1-1],
\end{tikzcd}
\end{equation}
where 
$\pi_1=\mathrm{Id}\ot\varepsilon$, $\pi_{2}=\varepsilon\ot\mathrm{Id}$ and the morphism $A\to\mathsf{Eq}(f)$ in $\mathsf{Hopf}_{\Bbbk,\mathrm{coc}}$ yielding the reflexivity of the relation is 
given by the corestriction of the comultiplication $\Delta_{A} \colon A \rightarrow A \otimes A$.

\begin{remark}\label{rmk:kernelpair}
We recall that, given morphisms $f:A\to C$, $g:B\to C$ in $\mathsf{Hopf}_{\Bbbk,\mathrm{coc}}$, the pullback of the pair $(f,g)$ is given by $(A\times_{C}B,\pi_{1},\pi_{2})$, where $A\times_{C}B$ is the Hopf subalgebra of $A\ot B$ defined by those elements $a\ot b\in A\ot B$ such that $ a_{1}\ot f(a_{2})\ot b=a\ot g(b_{1})\ot b_{2}$,
while $\pi_{1}=\mathrm{Id}\ot\varepsilon$, $\pi_{2}=\varepsilon\ot\mathrm{Id}$. 
Therefore, the kernel pair of a morphism $f:A\to B$ in $\mathsf{Hopf}_{\Bbbk,\mathrm{coc}}$ is given by $(\mathsf{Eq}(f),\pi_{1},\pi_{2})$, where $\mathsf{Eq}(f)$ is the Hopf subalgebra of $A\ot A$ defined by
\[
\mathsf{Eq}(f)=\{a\ot b\in A\ot A\ |\ a_{1}\ot f(a_{2})\ot b=a\ot f(b_{1})\ot b_{2}\}.
\]
\end{remark}

We can apply the functor 
\eqref{def:functorkernelunit} to the diagram \eqref{kernelpair} and obtain a reflexive relation 
\begin{equation}\label{EqRel'}
\begin{tikzcd}
	{[\mathsf{Eq}(f)]} && {[A]}
	\arrow[shift left=2, from=1-1, to=1-3, "{[\pi_{1}]}"]
	\arrow[shift right=2, from=1-1, to=1-3, "{[\pi_{2}]}"']
	\arrow[from=1-3, to=1-1]
\end{tikzcd}
\end{equation}
in $\mathsf{Hopf}_{\Bbbk,\mathrm{coc}}$. Since the latter is a semi-abelian category (hence, in particular, an exact Mal'tsev category), it follows that $([\mathsf{Eq}(f)], [\pi_1], [\pi_2])$ is an (effective) equivalence relation.
From the results in \cite{Bourn-Gran} and \cite{EVdL}, it follows that the normal subobject $[\mathrm{Hker}(f),A]\to A$ of $A$, where the commutator $[\mathrm{Hker}(f),A]$ is defined as in
\eqref{Huqcommutator}, can be obtained categorically by first ``normalizing'' the equivalence relation \eqref{EqRel'}
\[
\begin{tikzcd}\label{EqRel}
	[\mathsf{Ker}(f), A] \arrow[from=1-1, to=1-3, "{\mathsf{ker}([\pi_{1}])}"] && {[\mathsf{Eq}(f)]} && {[A].}
	\arrow[shift left=2, from=1-3, to=1-5, "{[\pi_{1}]}"]
	\arrow[shift right=2, from=1-3, to=1-5, "{[\pi_{2}]}"']
	\arrow[from=1-5, to=1-3]
\end{tikzcd}
\]
and then composing 
$[\pi_2] \mathsf{ker}([\pi_1 ]) \colon [\mathsf{Ker}(f), A] \rightarrow [A]$ 
with  $\mathsf{ker} (\eta_A) \colon [A] \rightarrow A$.

\section{Weak $\mathcal E$-universal normal extensions}
An important construction to define the fundamental group in the sense of categorical Galois theory \cite{Jan} is the one of the \emph{weak $\mathcal E$-universal normal extension} associated with any cocommutative Hopf algebra. 

\begin{definition}
    Let $B$ be a cocommutative Hopf algebra. A morphism $f:A\to B$ in $\mathcal{E}$ is called a \textit{weak $\mathcal{E}$-universal normal extension} of $B$ if it is a normal extension of $B$ such that for any other normal extension $g \colon C \rightarrow B $ of $B$ in $\mathcal{E}$ 
there exists a morphism $h \colon A \rightarrow C$ in $\mathsf{Hopf}_{\Bbbk,\mathrm{coc}}$ such that $g h = f$:
\[\begin{tikzcd}
	& C \\
	A & {B.}
	\arrow[from=1-2, to=2-2, "g"]
	\arrow[dashed, from=2-1, to=1-2, "h"]
	\arrow[from=2-1, to=2-2, "f"']
\end{tikzcd}\]
\end{definition}

In the next result, we prove the existence of a weak $\mathcal{E}$-universal normal extension for any cocommutative Hopf algebra.

\begin{proposition}
\label{WeakUniversal}
For any object $B$ in $\mathsf{Hopf}_{\Bbbk,\mathrm{coc}}$, there exists a weak $\mathcal E$-universal normal extension 
of $B$.
\end{proposition}
\begin{proof}
Given $B$ in $\mathsf{Hopf}_{\Bbbk,\mathrm{coc}}$, consider the canonical ``free presentation'' 
 $ \epsilon_B \colon FU(B) \rightarrow B$, where $\epsilon_B$
is the $B$-component of the counit $\epsilon$ of the adjunction \eqref{adj}.
One then considers the quotient $\pi \colon FU(B) \rightarrow \frac{FU(B)}{FU(B)[ \mathrm{Hker}(\epsilon_B), FU(B) ]^+}=:A$ in $\mathsf{Hopf}_{\Bbbk,\mathrm{coc}}$ 
, i.e. $\pi$ is the cokernel of the canonical inclusion $j:[\mathrm{Hker}(\epsilon_{B}),FU(B)]\hookrightarrow FU(B)$, which induces a unique morphism $f$ in $\mathsf{Hopf}_{\Bbbk,\mathrm{coc}}$ such that $f \pi = \epsilon_B$:
\[\begin{tikzcd}
	FU(B) && B \\
	& 
    A
	\arrow[from=1-1, to=1-3, "\epsilon_B"]
	\arrow[from=1-1, to=2-2, "\pi"']
	\arrow[from=2-2, to=1-3, "{f}"']
\end{tikzcd}\]
We know that $\epsilon_B$ is in $\mathcal E$ by Lemma \ref{lemma:counitadj} and then, since $f \pi = \epsilon_B$, we get that $f$ is in $\mathcal E$ by 4) of Lemma \ref{thm:classE}. 
It is also a normal extension, i.e. $[\mathrm{Hker}(f),A]=\Bbbk1_{A}$. Indeed, consider the following commutative diagram in $\mathsf{Hopf}_{\Bbbk,\mathrm{coc}}$
\[\begin{tikzcd}
	\mathbf{0} & \mathrm{Hker}(\epsilon_{B}) & FU(B) & B & \mathbf{0}\\
\mathbf{0} & \mathrm{Hker}(f) & A & B & \mathbf{0}
	\arrow[hook, from=1-2, to=1-3, "\mathsf{ker}(\epsilon_{B})"]
	\arrow[from=1-2, to=2-2, "\overline{\pi}"']
	\arrow[from=1-3, to=2-3, "\pi"]
	\arrow[from=1-3, to=1-4, "\epsilon_{B}"]
    \arrow[from=1-4, to=1-5, "{}"]
    \arrow[from=1-1, to=1-2, "{}"]
    \arrow[from=1-4, to=2-4, "{1_B}"]
    \arrow[from=2-4, to=2-5, "{}"]
    \arrow[from=2-1, to=2-2, "{}"]
    \arrow[from=2-3, to=2-4, "f"']
    \arrow[hook, from=2-2, to=2-3, "\mathsf{ker}(f)"']
\end{tikzcd}\]
where $\overline{\pi}$ is given by the restriction of $\pi$ to $\mathrm{Hker}(\epsilon_{B})$. The left-hand square is a pullback, since  the horizontal sequences are exact and the right-hand vertical morphism is a monomorphism (it is actually the identity on $B$).
It follows that $\overline{\pi}:\mathrm{Hker}(\epsilon_{B})\to\mathrm{Hker}(f)$ is a regular epimorphism,  since $\pi$ is a regular epimorphism and regular epimorphisms are pullback-stable in the semi-abelian category $\mathsf{Hopf}_{\Bbbk,\mathrm{coc}}$. Accordingly, we obtain $\pi(\mathrm{Hker}(\epsilon_{B}))=\mathrm{Hker}(f)$, and then
\[
[\mathrm{Hker} (f), A] = [ \pi (\mathrm{Hker} (\epsilon_B)), \pi (FU(B))] = \pi( [\mathrm{Hker} (\epsilon_B), FU(B)])=\Bbbk1_{A},
\]
where the second equality immediately follows since $\pi$ is a regular epimorphism and the Huq commutator is preserved by regular images in $\mathsf{Hopf}_{\Bbbk,\mathrm{coc}}$ \cite{GSV}. 
The (weak) universal property of $f$ in the statement then follows from the property of $\mathcal E$-projectivity of $FU(B)$ (see Lemma \ref{lem:projectiveobjects}) and the universal property of the Huq commutator (as recalled in \cite[pages 4182-4183]{GSV}, for instance). Indeed, given any other normal extension $g \colon C \rightarrow B$ of $B$ in $\mathcal{E}$, the $\mathcal E$-projectivity of $FU(B)$ yields a morphism $\overline{g} \colon FU(B) \rightarrow C$ such that $g \overline{g} = \epsilon_B$.
The universal property of the Huq commutator  then implies that there is a unique $\Tilde{g} \colon A \rightarrow C$ in $\mathsf{Hopf}_{\Bbbk,\mathrm{coc}}$ such that $  \Tilde{g} \pi = \overline{g}$. The morphism $\Tilde{g}$ is such that $g \Tilde{g} \pi = g \overline{g} = \epsilon_B = f \pi $, hence $g \Tilde{g} = f$, showing that $f$ is indeed a weak $\mathcal E$-universal normal extension of $B$.
\end{proof}

By adopting the categorical approach to Galois theory, one can define the Galois groupoid $\mathsf{Gal}(f)$ of $f$, where $f:A\to B$ is a weak $\mathcal{E}$-universal normal extension of a cocommutative Hopf algebra $B$, that always exists (thanks to the previous result). Consider the kernel pair $(\mathsf{Eq}(f),\pi_1 ,\pi_2 )$ of 
$f$ in $\mathsf{Hopf}_{\Bbbk,\mathrm{coc}}$
as in \eqref{kernelpair}. 

The \textit{Galois groupoid} $\mathsf{Gal}(f)$ of $f$ (with respect to the adjunction \eqref{Adj-abelianisation}) is defined as the image 
of $(\mathsf{Eq}(f),\pi_1 ,\pi_2 ) $
through the functor $\mathsf{ab} \colon \mathsf{Hopf}_{\Bbbk,\mathrm{coc}} \rightarrow \mathsf{Hopf}_{\Bbbk,\mathrm{coc}}^{\mathrm{com}}$:
\begin{equation}\label{Gal(f)} \begin{tikzcd}
	\mathsf{ab}(\mathsf{Eq}(f)) && {\mathsf{ab}(A).}
	\arrow[shift left=2, from=1-1, to=1-3, "\mathsf{ab}(\pi_1)"]
	\arrow[shift right=2, from=1-1, to=1-3, "\mathsf{ab}(\pi_2)"']
	\arrow[from=1-3, to=1-1]
\end{tikzcd}
\end{equation}
The fact that the reflexive graph \eqref{Gal(f)} is actually underlying a unique (internal) groupoid structure is due to the fact that $\mathsf{Hopf}_{\Bbbk,\mathrm{coc}}^{\mathrm{com}}$ is semi-abelian (see e.g. \cite[Theorem 8.3]{Gran}).
\begin{remark}
    In $\mathsf{Hopf}^{\mathrm{com}}_{\Bbbk,\mathrm{coc}}$, the morphisms $\mathsf{ab}(\pi_{i}):\mathsf{ab}(\mathsf{Eq}(f))\to\mathsf{ab}(A)$ are defined such that $\mathsf{ab}(\pi_i)\eta_{\mathsf{Eq}(f)}=\eta_{A}\pi_{i}$, where $\eta$ is the unit of the adjunction \eqref{Adj-abelianisation}. So they are induced by $\pi_{i}$ to the quotient $\mathsf{Eq}(f)/\mathsf{Eq}(f)[\mathsf{Eq}(f),\mathsf{Eq}(f)]^{+}$ and have codomain given by the quotient $A/A[A,A]^{+}$. We will usually denote $\mathsf{ab}(\pi_{i})$ as $\overline{\pi_{i}}$.
\end{remark}

In the next subsection, we explain how the Galois groupoid $\mathsf{Gal}(f)$ of a weak $\mathcal E$-universal normal extension $f \colon A \rightarrow B$ of $B$ allows one to establish an equivalence of categories between the category of normal $\mathcal E$-extensions of $B$ and a suitable category of discrete fibrations with codomain $\mathsf{Gal}(f)$. These observations are also based on a modified version of the categorical Galois theorem by G. Janelidze \cite{Jan-Algebra}, due to M. Duckerts-Antoine and T. Everaert \cite{DE}.

\subsection{The classification theorem for normal ${\mathcal E}$-extensions}\label{sec:classificationnormalextensions}
 First, we recall that a \emph{discrete fibration} of internal groupoids is an internal functor $(f_0, f_1, f_2) \colon H \rightarrow G$, depicted as 
\begin{equation}\label{Discrete}
\begin{tikzcd}
	H_2 && H_1 && H_0 \\
	\\
	G_2 && G_1 && G_0
	\arrow["m"{description}, from=1-1, to=1-3]
	\arrow[shift left=2, from=1-1, to=1-3,"{\pi_1}"]
	\arrow[shift right=2, from=1-1, to=1-3, "{\pi_2}"']
	\arrow[from=1-1, to=3-1, "f_2"']
	\arrow[shift left=2, from=1-3, to=1-5, "d"]
	\arrow[shift right=2, from=1-3, to=1-5,"c"']
	\arrow[from=1-3, to=3-3, "f_1"]
	\arrow["e"{description}, from=1-5, to=1-3]
	\arrow[from=1-5, to=3-5, "f_0"]
	\arrow["m"{description}, from=3-1, to=3-3]
	\arrow[shift left=2, from=3-1, to=3-3, "{\pi_1}"]
	\arrow[shift right=2, from=3-1, to=3-3, "{\pi_2}"']
	\arrow[shift left=2, from=3-3, to=3-5, "d"]
	\arrow[shift right=2, from=3-3, to=3-5, "c"']
	\arrow["e"{description}, from=3-5, to=3-3]
\end{tikzcd}
\end{equation}
where the upper part of the diagram represents the groupoid $H$ (where $(H_{2},\pi_{1},\pi_{2})$ is the pullback of the pair $(d,c)$ so that $H_2$ represents the ``object of composable morphisms''), the lower part the groupoid $G$, and the square 
$$\begin{tikzcd}
	{H_1} & {H_0} \\
	{G_1} & {G_0}
	\arrow[from=1-1, to=1-2, "c"]
	\arrow[from=1-1, to=2-1, "{f_1}"']
	\arrow[from=1-2, to=2-2, "{f_0}"]
	\arrow[from=2-1, to=2-2, "c"']
\end{tikzcd}$$
is a pullback. 

We define the category $\mathsf{DisFib}(\mathsf{Grpd}({\mathsf{Hopf}_{\Bbbk,\mathrm{coc}}^{\mathrm{com}}}),G)$ as the full subcategory of the slice category $\mathsf{Grpd}(\mathsf{Hopf}_{\Bbbk,\mathrm{coc}}^{\mathrm{com}}) \downarrow G$ whose objects are the discrete fibrations with domain an internal groupoid in $ \mathsf{Hopf}_{\Bbbk,\mathrm{coc}}^{\mathrm{com}}$ and codomain a fixed internal groupoid $G$. Moreover, we denote by $\mathsf{DisFib}_{\mathsf{Split}}(\mathsf{Grpd}({\mathsf{Hopf}_{\Bbbk,\mathrm{coc}}^{\mathrm{com}}}),G)$ the full subcategory of  $\mathsf{DisFib}(\mathsf{Grpd}({\mathsf{Hopf}_{\Bbbk,\mathrm{coc}}^{\mathrm{com}}}),G)$ whose objects are given by those discrete fibrations such that $f_0$, $f_1$ and $f_2$ are \emph{split epimorphisms}.

We are now going to show that there is an equivalence of categories between the category $$\mathsf{DisFib}_{\mathsf{Split}}(\mathsf{Grpd}({\mathsf{Hopf}_{\Bbbk,\mathrm{coc}}^{\mathrm{com}}}),\mathsf{Gal}(f)),$$
where $\mathsf{Gal}(f)$ is the Galois groupoid of a weak $\mathcal E$-universal normal extension $f:A\to B$ of $B$ (defined as in \eqref{Gal(f)}), and the category $\mathsf{Norm}(B)$ of normal $\mathcal E$-extensions of $B$, i.e. those normal extensions of $B$ that belong to $\mathcal{E}$. \medskip

We will denote by $\mathsf{SSpl}^{\mathcal E}(E,p)$ the set of the extensions $ \alpha \colon C \rightarrow B$ in $\mathcal E$ such that the pullback $p^*(\alpha)$ along the morphism $p \colon E \rightarrow B$ in $\mathcal{E}$ is a \emph{trivial extension} that is also a \emph{split epimorphism} in $\mathsf{Hopf}_{\Bbbk,\mathrm{coc}}$. 
The following result is a variation of \cite[Theorem 3.3]{DE}, which is suitable for our context:
\begin{proposition}\label{discrete-fib}
    Given $B$ in $\mathsf{Hopf}_{\Bbbk,\mathrm{coc}}$, we have an equivalence of categories 
    $$\mathsf{DisFib}_{\mathsf{Split}}(\mathsf{Grpd}({\mathsf{Hopf}_{\Bbbk,\mathrm{coc}}^{\mathrm{com}}}),\mathsf{Gal}(f)) \cong \mathsf{Norm}(B),$$
where $f$ is a weak $\mathcal E$-universal normal extension of $B$.
\end{proposition}
\begin{proof}
The proof consists of two main steps. 

In the first part, we prove that  
\[
\mathsf{Norm}(B)=\bigcup_{p \in \mathcal E}  \mathsf{SSpl}^{\mathcal E}(E,p) = \mathsf{SSpl}^{\mathcal E}(A,f),
\]
where $f \colon A \rightarrow B$ is a weak $\mathcal E$-universal normal extension of $B$, that always exists by Proposition \ref{WeakUniversal}.
 
To see this, let 
$\alpha \colon C \rightarrow B$ be any morphism in $\mathcal E$ belonging to $ \mathsf{SSpl}^{\mathcal E}(E,p)$, for some $p:E\to B$ in $\mathcal E$. By \cite[Proposition 2.4]{DE}, we have that $\alpha$ is a normal extension 
so that, building a weak $\mathcal E$-universal normal extension $f \colon A \rightarrow B$ as in Proposition \ref{WeakUniversal}, there exists a morphism $g \colon A \rightarrow C$ in $\mathsf{Hopf}_{\Bbbk,\mathrm{coc}}$
such that $\alpha g = f$. It follows that $f^* (\alpha) = g^* (\alpha^* (\alpha)) $ is both a trivial extension (since $\alpha^* (\alpha)$ is a trivial extension, and trivial extensions are pullback stable) and a split epimorphism (since $\alpha^* (\alpha)$ is a split epimorphism). As a consequence, $\alpha \in \mathsf{SSpl}^{\mathcal E}(A,f)$ and 

$$\bigcup_{p \in \mathcal E}  \mathsf{SSpl}^{\mathcal E}(E,p) \subseteq \mathsf{SSpl}^{\mathcal E}(A,f),$$
 so that
 $$\bigcup_{p \in \mathcal E}  \mathsf{SSpl}^{\mathcal E}(E,p) = \mathsf{SSpl}^{\mathcal E}(A,f).$$

The equality $$\mathsf{Norm}(B) = \bigcup_{p \in \mathcal E}  \mathsf{SSpl}^{\mathcal E}(E,p),$$ then follows from the following remarks: 1) any normal extension $g \colon C \rightarrow B$ in $\mathcal{E}$ belongs to $\mathsf{SSpl}^{\mathcal E}(A,f)$ for any weak $\mathcal E$-universal normal extension $f \colon A \rightarrow B$ (this follows from the fact that (split) trivial extensions are pullback stable); 2) any central extension is a normal extension, in our semi-abelian context \cite{JK}.

In the second part of the proof, we show that the categories $\mathsf{SSpl}^{\mathcal E}(A,f)$ (with obvious morphisms) and $\mathsf{DisFib}_{\mathsf{Split}}(\mathsf{Grpd}({\mathsf{Hopf}_{\Bbbk,\mathrm{coc}}^{\mathrm{com}}}),\mathsf{Gal}(f))$ are equivalent. We define a functor 
 \[
 F \colon \mathsf{SSpl}^{\mathcal E}(A,f)\rightarrow \mathsf{DisFib}_{\mathsf{Split}}(\mathsf{Grpd}({\mathsf{Hopf}_{\Bbbk,\mathrm{coc}}^{\mathrm{com}}}),\mathsf{Gal}(f)) 
 \]
which sends an object $c \colon C \rightarrow B$ in $\mathsf{SSpl}^{\mathcal E}(A,f)$ to the discrete fibration $$(\mathsf{ab}(p_1^1), \mathsf{ab}(p_1)) \colon (\mathsf{ab}(\mathsf{Eq}(p_2)) , {\mathsf{ab}({A \times_B C})) \rightarrow (\mathsf{ab}(\mathsf{Eq}(f)), \mathsf{ab}(A))}= \mathsf{Gal}(f)$$ depicted as 
\[
\begin{tikzcd}
	{\mathsf{ab}(\mathsf{Eq}(p_2)) \times_{{\mathsf{ab}({A \times_B C})}}  \mathsf{ab}(\mathsf{Eq}(p_2))} && \mathsf{ab}(\mathsf{Eq}(p_2)) && {\mathsf{ab}({A \times_B C})} \\
	\\
	{\mathsf{ab}(\mathsf{Eq}(f))  \times_{\mathsf{ab}(A)} \mathsf{ab}(\mathsf{Eq}(f)) } && \mathsf{ab}(\mathsf{Eq}(f)) && \mathsf{ab}(A)
	\arrow[from=1-1, to=1-3, ""]
	\arrow[shift left=2, from=1-1, to=1-3,""]
	\arrow[shift right=2, from=1-1, to=1-3, ""]
	\arrow[from=1-1, to=3-1]
	\arrow[shift left=2, from=1-3, to=1-5, "\mathsf{ab}(\pi'_1)"]
	\arrow[shift right=2, from=1-3, to=1-5,"\mathsf{ab}(\pi'_2)"']
	\arrow[from=1-3, to=3-3, "\mathsf{ab}(p_1^1)"]
	\arrow[from=1-5, to=1-3]
	\arrow[from=1-5, to=3-5, "\mathsf{ab}(p_1)"]
	\arrow[from=3-1, to=3-3]
	\arrow[shift left=2, from=3-1, to=3-3, " "]
	\arrow[shift right=2, from=3-1, to=3-3, ""]
	\arrow[shift left=2, from=3-3, to=3-5, "\mathsf{ab}(\pi_1)"]
	\arrow[shift right=2, from=3-3, to=3-5, "\mathsf{ab}(\pi_2)"']
	\arrow[from=3-5, to=3-3]
\end{tikzcd}
\]
where $(A \times _B C, p_1, p_2)$ is the pullback of $f$ and $c$, and $p^1_{1}:\mathsf{Eq}(p_{2})\to\mathsf{Eq}(f)$ is induced by the universal property of the pullback since $fp_{1}\pi'_{1}=cp_{2}\pi'_{1}=cp_{2}\pi'_{2}=fp_{1}\pi'_{2}$. Since the reflector $\mathsf{ab} \colon {\mathsf{Hopf}_{\Bbbk,\mathrm{coc}}} \rightarrow {\mathsf{Hopf}_{\Bbbk,\mathrm{coc}}^{\mathrm{com}}}$ preserves the pullback of $\pi_2$ along $p_1$, the internal functor $( \mathsf{ab}(p_1^1), \mathsf{ab}(p_1))$ is in fact a discrete split epimorphic fibration of internal groupoids with codomain $\mathsf{Gal}(f)$.

Let us first check that this functor $F$ is surjective on objects.
 Consider a split epimorphic discrete fibration 
 \[
\begin{tikzcd}
	{\mathsf{ab}(R)\times_{\mathsf{ab}(D) }  \mathsf{ab}(R)} && \mathsf{ab}(R) && {\mathsf{ab}(D)} \\
	\\
	{\mathsf{ab}(\mathsf{Eq}(f))  \times_{\mathsf{ab}(A)} \mathsf{ab}(\mathsf{Eq}(f)) } && \mathsf{ab}(\mathsf{Eq}(f)) && \mathsf{ab}(A)
	\arrow[from=1-1, to=1-3, ""]
	\arrow[shift left=2, from=1-1, to=1-3,""]
	\arrow[shift right=2, from=1-1, to=1-3, ""]
	\arrow[from=1-1, to=3-1]
	\arrow[shift left=2, from=1-3, to=1-5, ""]
	\arrow[shift right=2, from=1-3, to=1-5,""']
	\arrow[from=1-3, to=3-3, "r_1"]
	\arrow[from=1-5, to=1-3]
	\arrow[from=1-5, to=3-5, "r_0"]
	\arrow[from=3-1, to=3-3]
	\arrow[shift left=2, from=3-1, to=3-3, " "]
	\arrow[shift right=2, from=3-1, to=3-3, ""]
	\arrow[shift left=2, from=3-3, to=3-5, "\mathsf{ab}(\pi_1)"]
	\arrow[shift right=2, from=3-3, to=3-5, "\mathsf{ab}(\pi_2)"']
	\arrow[from=3-5, to=3-3]
\end{tikzcd}
\]
 in $\mathsf{DisFib}_{\mathsf{Split}}(\mathsf{Grpd}({\mathsf{Hopf}_{\Bbbk,\mathrm{coc}}^{\mathrm{com}}}),\mathsf{Gal}(f))$.
 By taking the pullback of $(r_1, r_0)$ along $(\eta_{\mathsf{Eq}(f)}, \eta_A)$ one gets the split epimorphic discrete fibration $$(\eta_{\mathsf{Eq}(f)}^*(r_1),\eta_{A}^*(r_0)) \colon (R,D) \rightarrow (\mathsf{Eq}(f), A).$$
 Note that the (internal) groupoid $(R,D)$ is necessarily an (internal) equivalence relation, since $(\eta_{\mathsf{Eq}(f)}^*(r_1),\eta_{A}^*(r_0))$ is a discrete fibration with codomain an equivalence relation, hence its domain $(R,D)$  has to be itself an  equivalence relation.  Since $f$ is an effective descent morphism in $\mathsf{Hopf}_{\Bbbk,\mathrm{coc}}$, there is a unique (up to isomorphism) morphism $\psi \colon D/R \rightarrow B$ corresponding to this discrete fibration, that is obtained by taking the quotient $q \colon D \rightarrow D/R$ of $D$ by the (effective) equivalence relation $R$.
 This morphism $\psi$ belongs to $\mathcal E$: indeed, since we know that the composite $ f \eta_A^*(r_0) = \psi q$
 belongs to $\mathcal E$ (by $3)$ of Lemma \ref{thm:classE}), it follows that $\psi \in \mathcal E$ (by $4)$ of Lemma \ref{thm:classE}).
The extension $\psi$ is also in $\mathsf{SSpl}^{\mathcal E}(A,f)$, since by construction $f^*(\psi)= \eta_A^*(r_0)$ is a trivial split extension.

To check that the functor $F$ is full, consider a morphism
\begin{equation}\label{morphism-discrete}
   \big((r_1,r_0) \colon (\mathsf{ab}(R), \mathsf{ab}(D)) \rightarrow \mathsf{Gal}(f)\big) \rightarrow  \big((s_1,s_0) \colon (\mathsf{ab}(S),\mathsf{ab}(E))\rightarrow \mathsf{Gal}(f)\big)
\end{equation}
  in $\mathsf{DisFib}_{\mathsf{Split}}(\mathsf{Grpd}({\mathsf{Hopf}_{\Bbbk,\mathrm{coc}}^{\mathrm{com}}}),\mathsf{Gal}(f))$.
As above, by pulling back the discrete fibrations along the units of the adjunction and by using the fact that $f$ is an effective descent morphism in $\mathsf{Hopf}_{\Bbbk,\mathrm{coc}}$, this morphism \eqref{morphism-discrete} induces  a morphism $$\alpha \colon (\psi \colon D/R \rightarrow B) \rightarrow (\phi \colon E/S \rightarrow B)$$ in $\mathsf{SSpl}^{\mathcal E}(A,f)$, that is such that  $F(\alpha)$ is the morphism in \eqref{morphism-discrete}, hence the functor $F$ is full. One then completes the proof by checking the faithfulness of the functor $F$. This latter property essentially follows from the fact that any (split) discrete fibration $ 
(s_1,s_0) : (S, E) \rightarrow  (\mathsf{Eq}(f) , A) 
 $
where $S = \eta_{\mathsf{Eq}(f)}^*(\mathsf{ab}(S) ) $ and $E = \eta_A^*(\mathsf{ab}(E))$
is such that
the pullback projections $S \rightarrow \mathsf{ab}(S)$ and $ S
 \rightarrow \mathsf{Eq}(f)$ are jointly monomorphic, as also are the pullback projections
 $E \rightarrow \mathsf{ab}(E)$ and $ E\rightarrow A$.
\end{proof}
 
Given a cocommutative Hopf algebra $B$, the normal $\mathcal E$-extensions of $B$ are so classified by those split epic discrete fibrations with codomain the Galois groupoid $\mathsf{Gal}(f)$ of a weak $\mathcal E$-universal normal extension $f:A\to B$ of $B$. We observe that a crucial role in this result is played by the existence of enough $\mathcal E$-projective objects (see Lemma \ref{lem:projectiveobjects}).

\section{Hopf formulae for 
cocommutative Hopf algebras}
In this section, we shall prove the Hopf formulae describing the fundamental group of a cocommutative Hopf algebra, by using the categorical approach developed in \cite{Jan} and the results concerning commutators in the category of cocommutative Hopf algebras obtained in \cite{GSV}.

The \textit{fundamental group} $\pi_1 (B)$ of an object $B$ in $\mathsf{Hopf}_{\Bbbk,\mathrm{coc}}$ (relative to the functor $\mathsf{ab} \colon \mathsf{Hopf}_{\Bbbk,\mathrm{coc}} \rightarrow \mathsf{Hopf}_{\Bbbk,\mathrm{coc}}^{\mathrm{com}}$) is the internal abelian group $\mathsf{Aut}_{\mathsf{Gal}(f)}(\mathbf{0})$ of the ``automorphisms of $\mathbf{0}$'' of the Galois groupoid $\mathsf{Gal}(f)$, where $f:A\to B$ is a weak $\mathcal{E}$-universal normal extension of $B$, defined via the following pullback: 
\[
\begin{tikzcd}
	\pi_1(B)  & \mathsf{ab}(\mathsf{Eq}(f)) \\
	 \mathbf{0} & \mathsf{ab}(A) \times \mathsf{ab}(A).
	\arrow[from=1-1, to=1-2]
	\arrow[from=1-1, to=2-1]
	\arrow["\lrcorner"{anchor=center, pos=0.125}, draw=none, from=1-1, to=2-2]
    \arrow[from=1-2, to=2-2, "{\langle\mathsf{ab}(\pi_1), \mathsf{ab}(\pi_2)\rangle}"]
	\arrow[from=2-1, to=2-2]
\end{tikzcd}
\]
We recall that the binary product in the category $\mathsf{Hopf}^{\mathrm{com}}_{\Bbbk,\mathrm{coc}}$ is given by the tensor product, so $\mathsf{ab}(A) \times \mathsf{ab}(A)=A/A[A,A]^{+}\ot A/A[A,A]^{+}$ and 
\[
\langle\overline{\pi_1}, \overline{\pi_2}\rangle:\frac{\mathsf{Eq}(f)}{\mathsf{Eq}(f)[\mathsf{Eq}(f),\mathsf{Eq}(f)]^{+}}\longrightarrow \frac{A}{A[A,A]^{+}}\ot\frac{A}{A[A,A]^{+}}
\]
is given by $(\overline{\pi_{1}}\ot\overline{\pi_{2}})\Delta_{\mathsf{ab}(\mathsf{Eq}(f))}$, where $\Delta_{\mathsf{ab}(\mathsf{Eq}(f))}$ is defined such that $$\Delta_{\mathsf{ab}(\mathsf{Eq}(f))}\eta_{\mathsf{Eq}(f)}=(\eta_{\mathsf{Eq}(f)}\ot\eta_{\mathsf{Eq}(f)})\Delta_{\mathsf{Eq}(f)},$$ where 
$\eta$ is the unit of the adjunction \eqref{Adj-abelianisation}. More precisely, given an element $$a\ot b+\mathsf{Eq}(f)[\mathsf{Eq}(f),\mathsf{Eq}(f)]^{+}$$ in $\mathsf{Eq}(f)/\mathsf{Eq}(f)[\mathsf{Eq}(f),\mathsf{Eq}(f)]^{+}$, we have:
\[
\begin{split}
    &\langle\overline{\pi_1},\overline{\pi_{2}}\rangle:a\ot b+\mathsf{Eq}(f)[\mathsf{Eq}(f),\mathsf{Eq}(f)]^{+}
    \longmapsto(a+ A[A,A]^{+})\ot(b+A[A,A]^{+}),
\end{split}
\]
i.e. $\langle\overline{\pi_1},\overline{\pi_{2}}\rangle\eta_{\mathsf{Eq}(f)}=\eta_{A}\ot\eta_{A}|_{\mathsf{Eq}(f)}$. By definition, $\pi_{1}(B)$ is the pullback of $u_{\mathsf{ab}(A)}\ot u_{\mathsf{ab}(A)}$ and $\langle\overline{\pi_1},\overline{\pi_{2}}\rangle$ in $\mathsf{Hopf}_{\Bbbk,\mathrm{coc}}^{\mathrm{com}}$. Therefore, $\pi_{1}(B)$ is given by those elements $\eta_{\mathsf{Eq}(f)}(a\ot b)$ in $\mathsf{ab}(\mathsf{Eq}(f))$, where $a\ot b$ is in $\mathsf{Eq}(f)$, such that
\[
1_{\mathsf{ab}(A)}\ot 1_{\mathsf{ab}(A)}\ot\eta_{\mathsf{Eq}(f)}(a\ot b)=\eta_{A}(a_{1})\ot\eta_{A}(b_{1})\ot\eta_{\mathsf{Eq}(f)}(a_{2}\ot b_{2}).
\]

\begin{remark}\label{invariant}
There is no ambiguity in defining the fundamental group $\pi_1(B)$ of a cocommutative Hopf algebra $B$ as 
$\mathsf{Aut}_{\mathsf{Gal}(f)}(\mathbf{0})$ for \emph{any} weak $\mathcal{E}$-universal normal extension $f$ of $B$. Indeed, given any other weak $\mathcal{E}$-universal normal extension $g$ of $B$, the kernel pairs $\mathsf{Eq}(f)$ of $f$ and $\mathsf{Eq}(g)$ of $g$ are equivalent (as internal equivalence relations), and this implies that  $\mathsf{Gal}(f)$ and $\mathsf{Gal}(g)$ are equivalent (as internal groupoids). It follows that 
$\mathsf{Aut}_{\mathsf{Gal}(f)}(\mathbf{0})$ is isomorphic to $\mathsf{Aut}_{\mathsf{Gal}(g)}(\mathbf{0})$, proving the ``invariance'' of $\pi_1(B)$. Therefore, the definition of $\pi_1(B)$ is independent from the choice of the weak $\mathcal{E}$-universal normal extension of $B$.
\end{remark}
 
We now give a better description of $\pi_{1}(B)$ in terms of a Hopf formula that is analogue to the well known one in the category of groups, by adapting the proofs given in \cite{Jan} to the present context. By adapting \cite[Theorem 2.1]{Jan} to our setting, we show the following result.

\begin{theorem}\label{thm:fundamentalgroup}
    Given a weak $\mathcal{E}$-universal normal extension $f \colon A \rightarrow B$ of a cocommutative Hopf algebra $B$, 
    the fundamental group $\pi_{1}(B)$ is isomorphic to the subobject $\mathsf{Ker} (f) \wedge [A,A]$, given explicitly by
\begin{equation}\label{Hopf-elements}
    \mathrm{Hker}(f)\cap[A,A]=\langle\{a_{1}b_{1}S(a_{2})S(b_{2})\in[A,A]\ |\ b_{1}a_{1}\ot f(a_{2}b_{2}S(a_{3})S(b_{3}))=ba\ot1\}\rangle_{A}.
\end{equation}
\end{theorem}
\begin{proof}
We sketch the proof for the reader's convenience, and we refer to \cite{Jan} for further details. 
We know that a weak $\mathcal E$-universal normal extension $f:A\to B$ of $B\in\mathsf{Hopf}_{\Bbbk,\mathrm{coc}}$ exists by Proposition \ref{WeakUniversal}. By definition, 
$\pi_1(B) $ can be obtained as the intersection of the kernels $\mathrm{Hker}(\overline{\pi_1})$ and $\mathrm{Hker}(\overline{\pi_2})$, where $\overline{\pi_{1}}=\mathsf{ab}(\pi_1)$, $\overline{\pi_{2}}=\mathsf{ab}(\pi_2)$ and $(\mathsf{Eq}(f),\pi_{1},\pi_{2})$ is the kernel pair of $f$ as in \eqref{kernelpair}. Accordingly, both squares in the following diagram are pullbacks:
\[
    \begin{tikzcd}
	\pi_1(B) & \mathrm{Hker}(\overline{\pi_{2}})  & 0 \\
	\mathrm{Hker}(\overline{\pi_1}) & \frac{\mathsf{Eq}(f)}{\mathsf{Eq}(f)[\mathsf{Eq}(f),\mathsf{Eq}(f)]^{+}} & \frac{A}{A[A,A]^{+}}
	\arrow[from=1-1, to=1-2]
	\arrow[from=1-1, to=2-1]
	\arrow[from=1-2, to=1-3]
	\arrow[hook, from=1-2, to=2-2]
	\arrow[from=1-3, to=2-3]
	\arrow[hook, from=2-1, to=2-2]
	\arrow[from=2-2, to=2-3, "\overline{\pi_2}"']
\end{tikzcd}
\normalsize
\]
This implies that the external part of the following diagram is also a pullback
\[\begin{tikzcd}
	\pi_1(B) & \mathrm{Hker}(\eta_A)  & 0 \\
	\mathrm{Hker}(f) & A & \frac{A}{A[A,A]^{+}}
	\arrow[from=1-1, to=1-2]
	\arrow[from=1-1, to=2-1]
	\arrow[from=1-2, to=1-3]
	\arrow[hook,from=1-2, to=2-2]
	\arrow[from=1-3, to=2-3]
	\arrow[hook,from=2-1, to=2-2,"\mathsf{ker}(f)"']
	\arrow[from=2-2, to=2-3, "\eta_A"']
\end{tikzcd}\]
where $\eta_{A}$ is the $A$-component of the unit $\eta$ of the adjunction \eqref{Adj-abelianisation}. The right-hand square is a pullback, hence so is the left-hand square, showing that $\pi_{1}(B)$ is isomorphic to the intersection $\mathrm{Hker}(f)\cap[A,A]$ (since $\mathrm{Hker}(\eta_A) = [A,A]$). In terms of elements, $\pi_{1}(B)$ is generated as an algebra by those elements $a_{1}b_{1}S(a_{2})S(b_{2})$, with $a,b\in A$, such that
\[
a_{1}b_{1}S(a_{2})S(b_{2})\ot f(a_{3}b_{3}S(a_{4})S(b_{4}))=a_{1}b_{1}S(a_{2})S(b_{2})\ot1.
\]
The latter equality is equivalent to $S(a_{1})S(b_{1})\ot f(a_{2}b_{2}S(a_{3})S(b_{3}))=S(a)S(b)\ot1$, and then to $b_{1}a_{1}\ot f(a_{2}b_{2}S(a_{3})S(b_{3}))=ba\ot1$ since the antipode of a cocommutative Hopf algebra is bijective. This proves that the fundamental group $\pi_1(B)$ can indeed be described as in \eqref{Hopf-elements}.
\end{proof}

 We can now give a characterization of the fundamental group $\pi_{1}(B)$ of a cocommutative Hopf algebra $B$, given an arbitrary $\mathcal{E}$-projective presentation of $B$ (see 2) in Definition \ref{def:EprojectiveHopf}). We observe that this is essentially proved in \cite{Jan}, although here we are using a different class $\mathcal E$ of extensions and of projective presentations. 

\begin{proposition}\label{Pi1}
Let $p \colon P \rightarrow B$ be any $\mathcal E$-projective presentation of a cocommutative Hopf algebra $B$. Then 
\[
\pi_1 (B) \cong \frac{\mathrm{Hker}(p) \cap [P,P]}{(\mathrm{Hker}(p) \cap [P,P])[\mathrm{Hker}(p),P]^{+}}.
\]
\end{proposition}
\begin{proof}
Consider the centralization $f:P/P[\mathrm{Hker}(p),P]^{+}\to B$ of the $\mathcal{E}$-projective presentation $p:P\to B$ of $B$, which yields a weak $\mathcal E$-universal normal extension $f$ of $B$ such that $f\pi=p$:
\[\begin{tikzcd}
	P && \frac{P}{P[\mathrm{Hker}(p),P]^{+}} \\
	& B
	\arrow[from=1-1, to=1-3, "\pi"]
	\arrow[from=1-1, to=2-2, "p"']
	\arrow[from=1-3, to=2-2,"f"]
\end{tikzcd}\]
as in the proof of Proposition \ref{WeakUniversal}. The quotient $\pi$ induces the following commutative diagram
\[
\footnotesize
\begin{tikzcd}
	 {{\mathrm{Hker}(p) \cap [P,P]}} && {\mathrm{Hker}(f) \cap\Big[\frac{P}{P[\mathrm{Hker}(p),P]^{+}}, \frac{P}{P[\mathrm{Hker}(p),P]^{+}}\Big]} \\
	& {[P,P ]} && {\Big[\frac{P}{P[\mathrm{Hker}(p),P]^{+}}, \frac{P}{P[ \mathrm{Hker}(p),P]^{+}}\Big]} \\
	{\mathrm{Hker}(p)} && {\mathrm{Hker}(f)} \\
	& P && \frac{P}{P[\mathrm{Hker}(p),P]^{+}}
	\arrow[dashed, from=1-1, to=1-3]
    \arrow[hook, from=1-3, to=3-3]
	\arrow[hook, from=1-1, to=2-2]
	\arrow[hook, from=1-1, to=3-1]
	\arrow[hook, from=1-3, to=2-4]
	\arrow[from=2-2, to=2-4]
	\arrow[hook, from=2-2, to=4-2]
	\arrow[hook, from=2-4, to=4-4]
	\arrow[from=3-1, to=3-3]
	\arrow[hook, from=3-1, to=4-2, "\mathsf{ker}(p)"']
	\arrow[hook, from=3-3, to=4-4, "\mathsf{ker}(f)"]
	\arrow[from=4-2, to=4-4, "\pi"']
\end{tikzcd}
\normalsize
\]
The left-hand and the right-hand faces of this parallelepiped are pullbacks by definition of intersection of subobjects, and the dotted arrow is induced by the universal property of the right-hand pullback (and is then the restriction of $\pi$ to $\mathrm{Hker}(p)\cap[P,P]$). One can easily check that all the squares in the parallelipiped are pullbacks, so that the semi-abelianness of the category $\mathsf{Hopf}_{\Bbbk,\mathrm{coc}}$ and the fact that $\pi$ is a regular epimorphism in $\mathsf{Hopf}_{\Bbbk,\mathrm{coc}}$ imply that also the other three horizontal morphisms are regular epimorphisms (equivalently, surjective maps) in $\mathsf{Hopf}_{\Bbbk,\mathrm{coc}}$. 
We observe that the inclusion $[\mathrm{Hker}(p),P]\subseteq\mathrm{Hker}(p)\cap[P,P]$ is obvious: given $a\in\mathrm{Hker}(p)$ and $b\in P$, we have
\[
\begin{split}
    a_{1}b_{1}S(a_{2})S(b_{2})\ot p(a_{3}b_{3}S(a_{4})S(b_{4}))&=a_{1}b_{1}S(a_{2})S(b_{2})\ot p(a_{3})p(b_{3})Sp(a_{4})Sp(b_{4})\\&=a_{1}b_{1}S(a_{2})S(b_{2})\ot p(b_{3})Sp(b_{4})\\&=a_{1}b_{1}S(a_{2})S(b_{2})\ot1.
\end{split}
\]
Accordingly, by taking into account the description of the fundamental group $\pi_1(B)$ given in Theorem \ref{thm:fundamentalgroup}, we have the isomorphisms
\[
\begin{split}
\pi_1(B)&\cong\mathrm{Hker}(f) \cap \Big[\frac{P}{P[\mathrm{Hker}(p),P]^{+}}, \frac{P}{P[\mathrm{Hker}(p),P]^{+}}\Big]\cong\frac{\mathrm{Hker}(p) \cap [P,P]}{(\mathrm{Hker}(p) \cap [P,P])[\mathrm{Hker}(p),P]^{+}} 
\end{split}
\]
where the last follows since the surjective dotted arrow is the restriction of $\pi$.
\end{proof}

Since the previous result holds for any $\mathcal{E}$-projective presentation of $B$, the fundamental group $\pi_1(B)$ provides an ``invariant'' of $B$. We will come back to this invariance in the next subsection.

\subsection{A Stallings-Stammbach exact sequence for Hopf algebras}
The notion of \emph{Baer invariant} in the context of semi-abelian categories with enough \emph{regular projectives} \cite{EVdL,EG} was shown to lead to some interesting results in homological algebra that we now establish in the category of cocommutative Hopf algebras. In this subsection, we explain how the categorical approach to Baer invariants can be adapted to cocommutative Hopf algebras by working with the class $\mathcal E$ of cleft extensions. 
The motivation for this approach precisely comes from the main observation that the semi-abelian category $\mathsf{Hopf}_{\Bbbk,\mathrm{coc}}$ has enough $\mathcal E$-projectives (see Lemma \ref{lem:projectiveobjects}). 

\begin{definition}
A functor $F \colon \mathsf{Pr}(\mathsf{Hopf}_{\Bbbk,\mathrm{coc}}) \rightarrow \mathsf{Hopf}_{\Bbbk,\mathrm{coc}}$ is called a \emph{Baer invariant} if, given two morphisms $(g,f)$ and $(g',f')$ in $\mathsf{Pr}(\mathsf{Hopf}_{\Bbbk,\mathrm{coc}})$ from $p:P\to B$ to $p':P'\to B'$ as in \eqref{morphismpresentations}, if $f=f'$ then $F((g,f))=F((g',f))$.
\end{definition}

The invariance (up to isomorphism) of the fundamental group proved in Proposition \ref{Pi1} 
implies the following result (by adapting the arguments in \cite[Proposition 4.1]{EG} to the present situation):
\begin{proposition}
    Given any $\mathcal E$-projective presentation $p:P\to B$
of a cocommutative Hopf algebra $B$, the quotient  
\[
  \frac{\mathrm{Hker}(p) \cap [P,P]}{(\mathrm{Hker}(p) \cap [P,P])[\mathrm{Hker}(p),P]^{+}}
\]
is a \emph{Baer invariant} of $B$. Accordingly, given another $\mathcal E$-projective presentation $q:Q\to B$ of $B$, 
the following two quotients are isomorphic in $\mathsf{Hopf}_{\Bbbk,\mathrm{coc}}$:
\[
  \frac{\mathrm{Hker}(p) \cap [P,P]}{(\mathrm{Hker}(p) \cap [P,P])[\mathrm{Hker}(p),P]^{+}}   \cong \frac{\mathrm{Hker}(q) \cap [Q,Q]}{(\mathrm{Hker}(q) \cap [Q,Q])[\mathrm{Hker}(q),Q]^{+}}. 
\]
\end{proposition}

In the following, we shall write  $\mathsf{H}_1$ for the functor sending $B$
to its abelianisation $\mathsf{H}_1(B) =\mathsf{ab}(B)$.
We also write $\mathsf{H}_2 \colon \mathsf{Hopf}_{\Bbbk,\mathrm{coc}} \rightarrow \mathsf{Hopf}_{\Bbbk,\mathrm{coc}} $ 
for the functor sending $B$ to the quotient 
\begin{equation}\label{Hopf-formula}
\mathsf{H}_2 (B) =  \frac{\mathrm{Hker}(p)\cap[P,P]}{(\mathrm{Hker}(p)\cap[P,P])[\mathrm{Hker}(p),P]^{+}},
\end{equation}
where $p:P\to B$ is an arbitrary $\mathcal{E}$-projective presentation of $B$. 

\begin{lemma}\label{Lemma-SS}
For any $\mathcal E$-projective presentation 
$p:P\to B$ of $B$ in $\mathsf{Hopf}_{\Bbbk,\mathrm{coc}}$, the following sequence in $\mathsf{Hopf}_{\Bbbk,\mathrm{coc}}$ 
\[\begin{tikzcd}
	\mathbf{0} & \mathsf{H}_2(B) & \frac{[P,P]}{[P,P][\mathrm{Hker}(p),P]^{+}} & B & \mathsf{H}_1(B)& \mathbf{0}
	\arrow[from=1-1, to=1-2]
	\arrow[from=1-2, to=1-3]
	\arrow[from=1-3, to=1-4]
	\arrow[from=1-4, to=1-5]
    \arrow[from=1-5, to=1-6]
\end{tikzcd}\]
is exact. 
\end{lemma}
\begin{proof}
Since $[\mathrm{Hker}(p),P]\subseteq\mathrm{Hker}(p)\cap[P,P]$, we have the inclusion $$[P,P][\mathrm{Hker}(p),P]^{+}\subseteq[P,P](\mathrm{Hker}(p)\cap[P,P])^{+}.$$ Then the ``double quotient isomorphism theorem''  that holds in any semi-abelian category (see for instance \cite[Lemma 1.3]{EG-monotone}) implies that the following sequence in $\mathsf{Hopf}_{\Bbbk,\mathrm{coc}}$
\begin{equation}\label{exseq1}
\begin{tikzcd}
	\mathbf{0}& \mathsf{H}_2(B) & \frac{[P,P]}{[P,P][\mathrm{Hker}(p),P]^{+}} & \frac{[P,P]}{[P,P](\mathrm{Hker}(p)\cap[ P,P])^{+}} & \mathbf{0}
	\arrow[from=1-1, to=1-2]
	\arrow[from=1-2, to=1-3] 
    \arrow[from=1-3, to=1-4]
    \arrow[from=1-4, to=1-5]
\end{tikzcd}
\end{equation}
is exact. One then considers the commutative diagram
\[\begin{tikzcd}
	  && {\mathrm{Hker}(p) \cap [P,P]} && {[P,P]} && {[B,B]} &&  \\
	\\
\mathbf{0}	 && \mathrm{Hker}(p) && P && B && \mathbf{0}
	\arrow[hook, from=1-3, to=1-5]
	\arrow[hook, from=1-3, to=3-3]
	\arrow[hook, from=1-5, to=3-5]
	\arrow[from=3-1, to=3-3]
	\arrow[hook, from=3-3, to=3-5]
    \arrow[from=3-7, to=3-9]
    \arrow[hook, from=1-7, to=3-7]
     \arrow[" ", from=1-5, to=1-7,"{[p]}"]
    \arrow[" ",from=3-5, to=3-7,"{p}"]
\end{tikzcd}\]
    where the left square is a pullback, the morphism $[p]:[P,P] \rightarrow [B,B] $ is a regular epimorphism (equivalently, a surjective map) in $\mathsf{Hopf}_{\Bbbk,\mathrm{coc}}$ (since so is $p$ and the functor $[\cdot] \colon \mathsf{Hopf}_{\Bbbk,\mathrm{coc}} \rightarrow \mathsf{Hopf}_{\Bbbk,\mathrm{coc}}$ defined as in \eqref{def:functorkernelunit} preserves regular epimorphisms), and the morphism $ [B,B] \hookrightarrow B$ is a monomorphism (equivalenty, an injective map) in $\mathsf{Hopf}_{\Bbbk,\mathrm{coc}}$. It follows that the upper sequence is also 
    exact, so $\mathrm{Hker}([p])=\mathrm{Hker}(p)\cap[P,P]$ and then the vector space kernel of $[p]$ is given by $\mathrm{ker}([p])=[P,P]\mathrm{Hker}([p])^{+}=[P,P](\mathrm{Hker}(p)\cap[P,P])^{+}$, where the first equality follows by the Newman's bijection \cite[Theorem 4.1]{Ne}.
    It then follows that the sequence 
\begin{equation}\label{exseq2}
\begin{tikzcd}
	\mathbf{0} & \frac{[P,P]}{[P,P](\mathrm{Hker}(p)\cap [P,P])^{+}}=
    \frac{[P,P]}{\mathrm{ker}([p])} & B & \mathsf{H}_1(B) & \mathbf{0}
    \arrow[from=1-1, to=1-2]
	\arrow[from=1-2, to=1-3]
	\arrow[from=1-3, to=1-4] \arrow[from=1-4, to=1-5]
\end{tikzcd}
\end{equation}
    is exact. Hence, from \eqref{exseq1} and \eqref{exseq2}, we get the thesis.
\end{proof}

Finally, we can provide a five-term exact sequence, starting from an arbitrary morphism in $\mathcal{E}$. This represents a “Hopf-theoretic” version of the classical \textit{Stallings-Stammbach 5-term exact sequence} that one has in the case of groups.

\begin{theorem}\label{SS}
Any short exact sequence in $\mathsf{Hopf}_{\Bbbk,\mathrm{coc}}$
\[\begin{tikzcd}\label{SES}
	\mathbf{0} & 
    \mathrm{Hker}(f) & A & B & \mathbf{0}
	\arrow[from=1-1, to=1-2]
	\arrow[hook, from=1-2, to=1-3]
	\arrow[from=1-3, to=1-4,"f"]
    \arrow[from=1-4, to=1-5],
\end{tikzcd}\]
where $f \in \mathcal E$, induces the five-term exact sequence in $\mathsf{Hopf}_{\Bbbk,\mathrm{coc}}$
\[\begin{tikzcd}
	\mathsf{H}_2(A) & \mathsf{H}_2(B) & \frac{\mathrm{Hker}(f)}{\mathrm{Hker}(f)[\mathrm{Hker}(f),A]^{+}} & \mathsf{H}_1(A) & \mathsf{H}_1(B) & \mathbf{0}
	\arrow[" ", from=1-2, to=1-3, ""]
	\arrow[" ",from=1-3, to=1-4," "]
    \arrow[" ", from=1-1, to=1-2,"\mathsf{H}_2(f) " ]
	\arrow[" ", from=1-4, to=1-5,"\mathsf{H}_1(f) " ]
    \arrow[" ",from=1-3, to=1-4," "]
    \arrow[" ",from=1-5, to=1-6," "].
\end{tikzcd}\]
The exact sequence depends naturally on the given exact sequence.
\end{theorem}
\begin{proof}
We already know that $\mathsf{H}_2(A)$ and $\mathsf{H}_2(B)$ do not depend on the choice of the $\mathcal{E}$-projective presentations of $A$ and $B$. Moreover, since morphisms in $\mathcal E$ are closed under composition (see 3) of Lemma \ref{thm:classE}), given an $\mathcal{E}$-projective presentation $p:P\to A$ of $A$, 
then $fp:P\to B$ is an $\mathcal{E}$-projective presentation of $B$. We clearly have that $\mathrm{Hker}(p)\subseteq\mathrm{Hker}(fp)$ and then $[\mathrm{Hker}(p),P]\subseteq[\mathrm{Hker}(fp),P]$.
In the following commutative diagram 
\[
\begin{tikzcd}
	\mathbf{0} & {\frac{[\mathrm{Hker}(fp),P]}{[\mathrm{Hker}(fp),P][\mathrm{Hker}(p),P]^{+}}} & \frac{[P,P]}{[P,P][\mathrm{Hker}(p),P]^{+}} & \frac{[P,P]}{[P,P][\mathrm{Hker}(fp),P]^{+}} & \mathbf{0}\\
\mathbf{0}	 & \mathrm{Hker}(f) & A & B & \mathbf{0}
	\arrow[from=1-1, to=1-2]
	\arrow[from=1-2, to=1-3]
	\arrow[dotted, from=1-2, to=2-2, "a"]
	\arrow[from=1-3, to=1-4]
	\arrow[from=1-3, to=2-3, "b"]
	\arrow[from=1-4, to=1-5]
	\arrow[from=1-4, to=2-4, "c"]
	\arrow[from=2-1, to=2-2]
	\arrow[hook, from=2-2, to=2-3]
	\arrow[from=2-3, to=2-4, "f"']
	\arrow[from=2-4, to=2-5]
\end{tikzcd}
\]
the upper sequence is exact by the ``double quotient isomorphism theorem'' (see \cite[Lemma 1.3]{EG-monotone}), the morphisms $b$ and $c$ are induced, respectively, by the restrictions of $p$ and $fp$ to $[P,P]$, and the morphism $a$ is induced by the universal property of the kernel of $f$. The image of $b$ is $[A,A]$ and the image of $c$ is $[B,B]$ by Lemma \ref{Lemma-SS}. To conclude, it suffices to show that the image of $a$ is $[\mathrm{Hker}(f),A]$, and then apply the Snake Lemma \cite{Bourn-Snake} to the diagram above.
Since $\mathsf{Hopf}_{\Bbbk,\mathrm{coc}}$ is semi-abelian, we have that $p(\mathrm{Hker}(fp))=\mathrm{Hker}(f)$ (by the same argument as in the proof of Proposition \ref{WeakUniversal}) and then $p([\mathrm{Hker}(fp),P])=[p(\mathrm{Hker}(fp)),A]=[\mathrm{Hker}(f),A]$. Therefore, the image of $a$ is $[\mathrm{Hker}(f),A]$ and the thesis follows.
\end{proof}

\medskip

\noindent\textbf{Acknowledgments}. This paper was written while Andrea Sciandra was member of the “National Group for Algebraic and Geometric Structures and their Applications” (GNSAGA-INdAM). He was also partially supported by the project funded by the European Union -NextGenerationEU under NRRP, Mission 4 Component 2 CUP D53D23005960006 - Call PRIN 2022 No. 104 of February 2, 2022 of Italian Ministry of University and Research; Project 2022S97PMY Structures for Quivers, Algebras and Representations (SQUARE).

This
research was supported by the Fonds de la Recherche Scientifique – FNRS under
Grant CDR No.J.0080.23.

\end{document}